\documentclass[a4paper,12pt]{amsart}

\usepackage{amssymb,xspace,amscd,yhmath,mathdots,txfonts,
mathrsfs,calligra}

\usepackage[matrix,arrow,curve]{xy}\CompileMatrices

\usepackage{hyperref}

\usepackage{yfonts}[1998/10/03]

\DeclareMathAlphabet{\mathpzc}{OT1}{pzc}{m}{it}

\DeclareMathAlphabet{\mathcalligra}{T1}{calligra}{m}{n}
\DeclareFontShape{T1}{calligra}{m}{n}{<->s*[2.2]callig15}{}

\numberwithin{equation}{section}

\theoremstyle{plain}

\newtheorem{thm}{Theorem}[section]

\newtheorem{lem}[thm]{Lemma}

\newtheorem{cor}[thm]{Corollary}

\newtheorem{prop}[thm]{Proposition}

\theoremstyle{definition}

\newtheorem{defn}{Definition}[section]
\newtheorem{exam}[thm]{Example}
\newtheorem{ntz}{Notation}[section]

\newtheorem{rmk}[thm]{Remark}

\newcommand\crdim{\mathrm{dim}_{CR}}
\newcommand\crcodim{\mathrm{codim}_{CR}}
\newcommand\Hb{\mathbb{H}}

\DeclareMathAlphabet{\mathpzc}{OT1}{pzc}{m}{it}

\DeclareMathOperator{\cL}{\mathcal{L}}

\newcommand\cO{\mathscr{O}}

\newcommand\so{\Tilde{\mathscr{O}}}

\newcommand\sO{\Tilde{\mathscr{O}}_{M}}
\newcommand\tiO{\Tilde{\mathscr{O}}}
\newcommand\sfE{\textsf{E}}

\DeclareMathOperator{\cW}{\mathcal{W}}
\DeclareMathOperator{\cU}{\mathcal{U}}
\DeclareMathOperator{\R}{{\mathbb{R}}}

\DeclareMathOperator{\cC}{{\mathcal{C}}}
\DeclareMathOperator{\rank}{rank}

\DeclareMathOperator{\N}{\mathrm{N}}
\DeclareMathOperator{\Z}{\mathbb{Z}}

\DeclareMathOperator{\im}{{\mathrm{Im}}}
\DeclareMathOperator{\re}{{\mathrm{Re}}}

\DeclareMathOperator{\C}{\mathbb{C}}

\newcommand\GL{\mathbf{GL}}

\newcommand\sfB{\textsf{B}}

\newcommand\sq{\mathpzc{s}}
\newcommand\rr{\mathpzc{r}}
\newcommand\tq{\mathpzc{t}}

 \newcommand\op{\mathrm{p}_0}

\DeclareMathOperator{\D}{\mathrm{D}}

\newcommand\CP{\mathbb{CP}}

\newcommand\Gr{\textrm{Gr}}

\newcommand\sfX{\mathsf{X}}
\newcommand\sfA{\mathsf{A}}

\newcommand\sfY{\mathsf{Y}}

\newcommand\Id{\mathrm{I}}

\newcommand\Ci{\mathcal{C}^\infty}

\newcommand\cS{\mathcal{S}}

\newcommand\Jd{\mathrm{J}}

\newcommand\rk{\mathrm{rank}}

\newcommand\pct{\mathrm{p}}
\newcommand\qct{\mathrm{q}}

\newcommand\iq{\mathpzc{i}}
\newcommand\pq{\mathpzc{p}}
\newcommand\qq{\mathpzc{q}}
\newcommand\cH{\mathcal{H}}
\newcommand\cF{\mathcal{F}}

\newcommand\Cg{\mathfrak{C}}
\newcommand\cc{\textfrak{C}}

\newcommand\Og{\mathcal{O}}

\newcommand\zq{\mathpzc{z}}

\newcommand\cI{\mathscr{I}}
\newcommand\cN{\mathcal{N}}

\newcommand\vq{\mathpzc{v}}
\newcommand\wq{\mathpzc{w}}
\newcommand\uq{\mathpzc{u}}

\newcommand\rtt{\texttt{r}}

\newcommand\Zz{\textgoth{Z}}
\newcommand\cE{\mathcal{E}}
\newcommand\codim{\mathrm{codim}}

\newcommand\sfM{\textsf{M}}

   \def\DHLhksqrt#1#2{\setbox0=\hbox{$#1\sqrt{#2\,}$}\dimen0=\ht0
     \advance\dimen0-0.2\ht0
     \setbox2=\hbox{\vrule height\ht0 depth -\dimen0}
     {\box0\lower0.4pt\box2}}

\title[$LACR$ functions, maximum principle and holomorphic extension]
{Locally approximable $CR$ functions, a sharp maximum modulus principle
and holomorphic extension}

\author[M.~Nacinovich]{Mauro Nacinovich}
\address{Mauro Nacinovich:
Dipartimento di Matematica\\ II Universit\`a di Roma
``Tor Ver\-ga\-ta''\\ Via della Ricerca Scientifica\\ 00133 Roma
(Italy)}
\email{nacinovi@mat.uniroma2.it}
\author[E.~Porten]{Egmont Porten}
\address{Egmont Porten: Department of Mathematics\\ 
Mid Sweden University\\ 85170 Sundsvall (Sweden)}
\email{Egmont.Porten@miun.se}

\date{today}

\subjclass[2000]{Primary: 32V10
Secondary: 32V05, 32V25, 32V30,  32D10, 32D20, 14J17}
\keywords{$CR$-embedding, weak pseudoconcavity, maximum modulus principle, $CR$ singularity,
holomorphic extension}

\date{\today}
\begin{document}
\maketitle

\begin{abstract}
We introduce a notion of \textit{locally approximable continuous $CR$ functions}
on locally closed subsets of reduced complex spaces, 
generalizing both holomorphic functions and $CR$ functions on $CR$ submanifolds. 
Under additional assumptions of \textit{set-theoretical weak pseudoconcavity} 
we prove optimal maximum modulus principles for these functions. 
Restricting to real submanifolds (possibly with CR singularities) of complex manifolds,
we generalize results on holomorphic extension known for $CR$ submanifolds.
\end{abstract}
\tableofcontents

\maketitle

\section*{Introduction}
Since 19th century, maximum modulus principles and 
extension of holomorphic functions
have been persistent themes in complex analysis. 
More recently, these topics 
reemerged 
in the realm of the study of $CR$ functions on 
$CR$ manifolds, with increasingly close connections  
with
envelopes of holomorphy. \par
In this paper we approach  
these issues 
in a way 
which applies 
to more general subsets of complex manifolds, in particular to real submanifolds 
with $CR$ singularities.
A strong motivation  
for allowing
$CR$ singularities are topological obstructions 
to embedding smooth compact $CR$ manifolds into complex Stein manifolds 
and to 
removing 
local $CR$ singularities by small deformations.
\par
As a point of departure, we introduce the sheaf ${\tiO_{\!\sfA}}$
of \textit{locally approximable continuous $CR$ functions}, 
$LACR$  
for short,
on locally closed subsets~$\sfA$ of reduced complex spaces. 
By the Baouendi-Treves approximation theorem, we get back the usual continuous $CR$ functions 
if $\sfA$ is a smooth $CR$ manifold.
In general, one needs \emph{repeated} uniform local approximations
in order to obtain 
the Fr\'echet sheaf 
${\tiO_{\!\sfA}},$
its definition requiring 
a general construction by transfinite induction, which was pioneered by Rickard
(\cite{Ric}).
The construction abstractly  requires several approximation steps and for us
the general question on the ordinal number of needed steps 
remains open, although we get some
partial results and we conjecture that 
countably many are sufficient when $\sfA$ is a manifold.
\par
In some cases 
$\tiO_{\sfA}$ may even coincide with the sheaf of continuous functions
(by Stone--Weierstrass  theorem this is the case when $\sfA$ is a totally real
submanifold). 
Thus in our investigation 
we  turn to geometric conditions on $\sfA$ ensuring that sections of ${\tiO_{\!\sfA}}$
share some properties of the 
holomorphic functions.
Following the work of several authors building on seminal work \cite{Ro57} by Rothstein,
we introduce and study the notions   of set-theoretic 
$\pq$-pseudoconvexity and $\qq$-pseudoconcavity.
As our first major results we obtain in  
\S\ref{sec-max}, \S\ref{sec-4}
sharp maximum modulus principles
on \emph{set-theoretically weakly $1$-pseudoconcave} locally closed sets. 
One of the attractive features of set-theoretic $\qq$-pseudoconcavity is
that it  
behaves 
more naturally 
under
slicing by complex objects than the
corresponding notion for $CR$ manifolds, 
yielding 
in particular 
immediate corollaries 
for real submanifolds 
carrying 
$CR$ singularities.
\par
The maximum modulus principle being established for set-theoretically weakly pseudoconcave sets,
we turn to extension theorems.
It is known that there is a close connection between maximum modulus principles
and local holomorphic extension of $CR$ functions from $CR$ manifolds.
Our general setting requires new methods,
since the domains for successive approximations may depend on the individual $LACR$ function.
We will restrain our attention first to submanifolds $\sfM\,{\subset}\, \sfX$ 
carrying isolated $CR$ singularities,
a situation appearing most naturally for submanifolds of codimension $2,$
motivating considerable recent research activity (see e.g. \cite{Coff, Gar, LNR}).
We investigate
local holomorphic extension to full neighbourhoods in $\sfX.$ 
For $CR$ manifolds this problem is intimately linked to regularity theory and to weak pseudoconcavity
(see \cite{NP2015} and references therein),
the general intuition being that extension to full neighbourhoods 
(or the equivalent concept of $CR$-hypoellipticity)
should be equivalent to weak pseudoconcavity complemented by a certain amount of nonintegrability
of the holomorphic tangent bundle.
Here we prove that set-theoretic weak pseudoconcavity at an isolated $CR$ singularity is a sufficent condition,
provided $M$ has the local holomorphic extension property at all sufficiently 
close 
nearby points. 
\par
We generalize this result to CR singularities of positive dimension satisfying appropriate geometric properties.
\par
In the latter case it is no longer immediate how to construct examples.
In the final section, we approach the question via a general study of the Gauss map
of a smooth real submanifold in complex space.
This gives control on nonisolated CR singularities
and insight into the question to which extent our examples are stable 
under appropriate small deformations.
As a byproduct, we extend transversality results of \cite{Cof2009,Gar} 
to submanifolds of codimension larger than $2$.
%In the final section we study local $CR$ singularities of smooth submanifolds,
%extending to higher codimension transversality results of \cite{Cof2009,Gar}
%and use them to construct examples to which the results of the previous sections apply.

\section{Set-theoretical 
pseudo-convexity and concavity} \label{sec-1}
Let $f_1,\hdots,f_m$ be holomorphic functions on an
open subset $U$ of $\C^N.$
The set 
\begin{equation} \label{eq-1.1}
 Z{=}\{\zq{\in}{U}\mid {f}_1(\zq){=}0,\hdots,f_m(\zq){=}0\}
\end{equation}
of their 
common zeros  is called
an \textit{analytic subvariety} of $U.$ Denote by 
$\cO_{U}$ the sheaf of germs of holomorphic functions
on $U$ and by $\cI_{Z}$ its ideal subsheaf generated by
$f_1,\hdots,f_m.$ The quotient sheaf $\cO_{U}/\cI_{Z}$
induces  a sheaf
of rings $\cO_{Z}$ 
on $Z.$ 
The \emph{ringed space} $(Z,\cO_{Z})$ is a
\emph{local model}. 
 A \emph{complex space} (see, e.g. \cite{fischer1976}) 
is a Hausdorff 
ringed space $(\sfX,\cO_\sfX)$ such that each point 
$\pct$ of $\sfX$ has a neighbourhood
$\omega_{\pct}$ with 
$(\omega_{\pct},\cO_{\sfX}|_{\omega_{\pct}})$ 
isomorphic  to a local model. 
We say that the local model $(Z,\cO_Z)$ is \emph{reduced} if 
$\cI_{\!Z}$ consists of the germs 
of holomorphic functions vanishing on~$Z$ 
and call $(\sfX,\cO_{\sfX})$
\emph{reduced} if all its local models are reduced. 
This is equivalent to the fact that the ideal
$\cN_{\sfX}$ 
of nilpotent germs of $\cO_{X}$
is trivial. The dimension of $(\sfX,\cO_{\sfX})$ at a point
$\pct\,{\in}\,\sfX$ is the Krull dimension of the local ring 
$\cO_{\sfX,\pct}.$ 
In the following, while considering 
a reduced complex space 
$(\sfX,\cO_{\sfX})$ of pure dimension $n,$
we 
will often simplify notation by omitting
to indicate the structure
 sheaf $\cO_{\sfX}.$ \par 
 For the 
 classical
 notion of 
 $\qq$-pseudoconvexity/pseudoconcavity 
 for smooth complex manifolds
 we refer e.g. 
 to \cite{demailly1997}.
\par 
Let $\sfA$ be a  
locally 
closed subset of a reduced complex space 
$\sfX$ of pure dimension~$n.$
\begin{lem}\label{lem-1.1}
For a point
 $\pct$  
 of $\sfA$ 
the following are equivalent: 
\begin{enumerate}
\item There is 
an open neighbourhood $\omega_{\pct}$ of $\pct$ in $\sfX,$
a  
local isomorphism  $$\zq:(\omega_{\pct},\cO_{\sfX}|_{\omega_{\pct}})
\to(Z,\cO_{Z})$$ with a local model 
\eqref{eq-1.1},
and a smooth function $\rhoup\in\Ci(U,\R)$ such that,
for some open subset $\omegaup'_{\pct}$ with $\pct\,{\in}\,\omegaup'_{\pct}\,{\subseteqq}\,\omegaup_{\pct},$ 
\begin{equation} \label{eq-1.2}
\begin{cases} 
\zq(A\,{\cap}\,\omegaup'_{\pct})\,{\subseteq}\,\{\rhoup(\zq){\leq}\rhoup(\zq(\pct))\},
\\
\iq^{{-1}}\partial\bar\partial\rhoup(\zq(\pct))\;\; 
\text{has at least
 $N{-}\qq{+}1$ positive eigenvalues.}
 \end{cases}
\end{equation}
\item For every local model
\eqref{eq-1.1} at $\pct,$ we can find a smooth function 
$\rhoup\in\Ci(U,\R)$  for which \eqref{eq-1.2}
holds true. 
\end{enumerate}
 \end{lem} 
\begin{proof}
Assume that we have two local models at $\pct$ 
and that the first 
one 
satisfies \eqref{eq-1.2} with $\omegaup'_{\pct}\,{=}\,\omegaup_{\pct}.$ 
We can assume that they are defined on the same open
neighbourhood $\omega_{\pct}$ of $\pct$ in $\sfX$ by functions 
$\zq_{i}\,{=}\,\phiup_{i}\,{\in}\,\cO_{\sfX}(\omega_{\pct})$ ($1{\leq}i{\leq}N$)
and $\zq'_{i}\,{=}\,\psiup_{i}\,{\in}\,\cO_{\sfX}(\omega_{\pct})$ ($1{\leq}i{\leq}N'$),
yielding the respective isomorphisms 
with the local models. 
Let us consider the immersion  
\begin{equation*}
 \zetaup:\omega_{\pct}\ni\qct \to (\phiup_{1}(\qct),\hdots,\phiup_{N}(\qct),
 \psiup_{1}(\qct),\hdots,\psiup_{N'}(\qct))\in U\times{U}'\subset\C^{N+N'}.
\end{equation*}\par
The pullbacks to $Z$ 
of the coordinate functions 
$\zq'_1,\hdots,\zq'_{N'}$ belong to 
$\cO_{Z}(Z)$ and are hence, on a neighbourhood of
$\zq(\pct)$ in $Z,$  restrictions of 
functions $g_1,\hdots,g_{N'},$ defined  and holomorphic 
on an open neighourhood of $\zq(\pct)$ in $U.$
By shrinking, we may assume 
that
they are defined
on $U.$ 
\par
Then $\zq'_1{-}g_1(\zq),\hdots,\zq'_{N'}{-}g_{N'}(\zq)$ 
are holomorphic on $U{\times}U'$ and
vanish on 
$Z''\,{=}\,\{(\phiup(\qct),\psiup(\qct))\mid \qct\,{\in}\,\omega_{\pct}\}
\,{\subset}\,{U}\,{\times}\,{U}'.$ 
Since we
assumed that \eqref{eq-1.2} holds for the local model $(Z,\cO_Z),$ 
then, 
 the complex Hessian of  
\begin{equation*}
 \tilde{\rhoup}(\zq,\zq')=\rhoup(\zq){+}
 {\sum}_{i{=}1}^{N'}|\zq_i'-g_i(\zq)|^2
\end{equation*}
 has at $(\zq(\pct),\zq'(\pct))$ 
 at least $N{+}N'{-}\qq{+}1$ positive eigenvalues. 
\par 
As we did for the $\zq'_i,$ after shrinking we can assume that
the pullbacks of the coordinates $\zq_{i}$ to $Z'$ are restrictions
of holomorphic functions $g'_{1},\hdots,g'_{N}$ which are defined on $U'.$
They yield holomorphic functions 
 $z_i{-}g'_i(\zq')$ on $U\,{\times}\,U'$ that
 vanish on $Z'',$
for $1{\leq}i{\leq}N.$ 
The map $$U'{\ni}\zq'{\to}(g'(\zq'),\zq'){\in}{U}{\times}U'$$
is a holomorphic immersion mapping $Z'$ to $Z''.$ Then
the pullback of $\tilde{\rhoup}$ to $U'$ 
satisfies \eqref{eq-1.2} for the local model $(Z',\cO_{Z'}).$ 
\end{proof}
Based on similar arguments, one can intrinsically introduce
the notions of smooth functions and of smooth $\qq$-plurisubharmonic functions 
on complex spaces, see e.g. \cite{demailly1997}.
\begin{rmk}
 Condition \eqref{eq-1.2} is equivalent to 
 \begin{equation} \label{eq-1.2'} \tag{\ref{eq-1.2}$'$}
\begin{cases} 
\zq((A\,{\cap}\,\omegaup'_{\pct})\,{\backslash}\{\pct\})\,{\subset}\,\{\rhoup(\zq){<}\rhoup(\zq(\pct))\},
\\
\iq^{{-1}}\partial\bar\partial\rhoup(\zq(\pct))\;\; 
\text{has at least
 $N{-}\qq{+}1$ positive eigenvalues.}
 \end{cases}
\end{equation}
Indeed \eqref{eq-1.2} is still satisfied on a smaller neighbourhood of $\pct$ after
substituting $\rhoup$ by \;
$\qct\,{\mapsto}\,\left(\rhoup(\qct)\,{-}\,\epsilonup{\sum}_{i=1}^{N}|\zq_{i}(\qct)-\zq_{i}(\pct)|^{2}\right)$\;
for $0{<}\epsilonup{\ll}1.$ 
\end{rmk}

\begin{defn}\label{defn1.1a} Let $\pct\,{\in}\,\sfA.$ 
We say that $\sfA$ is 
\emph{set-theoretically $\qq$-pseudoconvex at $\pct$} 
if  
conditions
\eqref{eq-1.2} in Lemma~\ref{lem-1.1} are verified. 

\par
We say that $\sfA$ is 
\emph{set-theoretically weakly $\qq$-pseudoconcave at $\pct$}
when it is not set-theoretically 
\mbox{$\qq$-pseudoconvex} at $\pct$.\par
We will drop the clause 
{\textquotedblleft{at~$\pct$}\textquotedblright} 
when the conditions hold at all
points. 
\end{defn}

\begin{rmk} \label{rmk-1.2}
We can define a category whose objects are  
locally 
closed
continuous 
embeddings
$(\imath:\sfA\hookrightarrow\sfX)$ 
of an abstract
Hausdorff space $\sfA$  into a reduced complex
space 
$\sfX$ 
and identify 
 two elements $(\imath:\sfA\hookrightarrow\sfX)$ and
$(\imath':\sfA'\hookrightarrow\sfX')$ if there is a third  
$(\imath'':\sfA''\hookrightarrow\sfX'')$,
homeomorphisms $\sfA\leftrightarrow \sfA''$, 
$\sfA'\leftrightarrow{\sfA''}$ and 
holomorphic immersions $\jmath:X''\hookrightarrow X$
and $\jmath':X''\hookrightarrow X'$, which make the 
following
diagram commute 
\begin{equation*}
 \xymatrix{ & \sfA \ar[rr]^{\imath}&& \sfX\\
 \sfA'' \ar[rr]_{\imath''} \ar@{<->}[ur] \ar@{<->}[dr] 
 && \sfX'' \ar[ur]_{\jmath'} \ar[dr]^{\jmath'}\\
 & \sfA' \ar[rr]^{\imath'} && \sfX'.
 }
\end{equation*}
One can easily check that this 
is in fact an equivalence relation. 
The objects of the quotient category 
can be considered as \textit{abstract  
locally 
closed subsets
of reduced complex spaces}.
\end{rmk}
By Lemma~\ref{lem-1.1}, if $(\sfA\,{\hookrightarrow}\,\sfX)$ and $(\sfA'\,{\hookrightarrow}\,\sfX')$
 are equivalent objects of the category of
 Remark~\ref{rmk-1.2}, then $\sfA$ is set-theoretically 
$\qq$-pseudoconvex/pseu\-do\-con\-cave
at a point $\pct\in{\sfA}$ iff $\sfA'$ 
enjoys the same property at the corresponding point $\pct'$ of 
${\sfA}'.$
This was the reason 
for not explicitly mentioning  
the \textit{ambient space} $\sfX$ in
Definition~\ref{defn1.1a}.

\par \smallskip
For example, if $\sfX$ is an $n${-}dimensional complex manifold,
the closure of a strictly 
pseudoconvex
domain $D\Subset
\sfX$ is set-theoretically $\qq$-pseudoconvex, 
for all $1{\leq}{\qq}{\leq}{n},$ 
at each $\pct\,{\in}\,\partial{D}$
and weakly $n$-pseudo\-con\-cave at each  ${\pct\,{\in}\, D}$.   
\subsection{Complex subvarieties and smooth real submanifolds of complex manifolds}
Let $\sfA$ be an analytic subset of the reduced complex space $\sfX$ 
and $\pct$ a point of $\sfA$. If $\dim_{\C}(\sfA,\pct)\,{=}\,d,$ then $\sfA$ is
weakly set-theoretically pseudoconcave at $\pct$ for all $\qq{\leq}d.$ 
If, moreover, $\pct$ is a regular point of $\sfA,$ then $\sfA$ is set-theoretically
$\pq$-pseudoconvex at $\pct$ for all $\pq{>}d.$ 

\par
Vice versa, we have:
\begin{lem}\label{lem1.4}
 Let $\sfA$ be a $\cC^2$-smooth submanifold of 
 a complex manifold $\sfX,$ with complex structure 
 $\Jd\,{:}\mathrm{T}\sfX\,{\to}\,\mathrm{T}\sfX.$ 
 If $\sfA$ is set-theoretically 
 weakly $\qq$-pseudoconcave at $\pct,$ then 
 $\dim_{\R}(\mathrm{T}_{\pct}\sfA\cap\Jd(\mathrm{T}_{\pct}\sfA))\geq{2\qq}.$ 
 In particular, the real dimension of $\sfA$ is at least $2\qq$ and 
a real $2\qq${-}dimensional submanifold $\sfA$ of $\sfX$
which is  set-theoretically 
 weakly $\qq$-pseu\-do\-con\-cave at all points is a complex submanifold
 of complex dimension $\qq$ of $\sfX.$ 
\end{lem} 
\begin{proof} Let $m{=}\dim_{\R}(\sfA).$ 
 We can assume that $\sfX$ is an open ball in $\C^n$, 
 centered at $\pct\,{=}\,0\,{\in}\,{\sfA},$ and 
 $\sfA\,{=}\,\{\rhoup_1\,{=}\,0,\hdots,\rhoup_d\,{=}\,0\},$ with $d\,{=}\,2n{-}m,$ 
 for smooth real valued functions $\rhoup_1,\hdots,\rhoup_d
 \in\Ci(\sfX)$ with $d\rhoup_1(z)\wedge\cdots\wedge{d}\rhoup_d(z)\neq{0}$ for $z\in\sfX.$  
We have  
\begin{equation*}
 \sfA\subseteq\left\{r_{\epsilonup}=\rhoup_{1}
 +\epsilonup^{{-1}}{\sum}_{{i=1}}^{d}\rhoup_{i}^{2}-\epsilonup\|\zq\|^{2}\,{\leq}\,0
 \right\}
 \subset\sfX
\end{equation*}
 for all ${\epsilon}{>}0.$ For ${\epsilon}{>}0$ sufficiently small, the complex Hessian of
$r_{\epsilon}$ is positive-definite on the 
$\C$-linear span of 
$\partial\rhoup_1,\hdots,\partial\rhoup_d,$ which has complex dimension 
$n\,{-}\dim_{\C}(\mathrm{T}_{\pct}\sfA\,{\cap}\,\Jd(\mathrm{T}_{\pct}\sfA)).$ 
This observation yields the statement.
\end{proof}
By a classical result of Hartogs 
the last statement of 
Lemma\,\ref{lem1.4}
is still valid for 
\emph{continuous} submanifolds.
\begin{cor}\label{cor1.4}
A complex analytic
subvariety of pure dimension $\qq$ of   a complex
manifold $\sfX$ is set-theoretically weakly $\qq$-pseudoconcave.
\end{cor} 
\begin{proof} Let $\sfA\,{\subset}\,\sfX$ be a complex subvariety of $\sfX$ and $\pct$ a point of $\sfA.$ 
Assume by contradiction that there is a smooth real valued $\rhoup,$ defined on a neighbourhood
$U$ of $\pct$ in $\sfX$ 
whose
Hessian $\iq^{-1}\partial\bar{\partial}\rhoup(\pct)$
has 
at least $n{-}\qq{+}1$ positive eigenvalues, vanishing at $\pct$ and
 such that 
$\sfA\,{\backslash}\{\pct\}\,{\subset}\,\{\rhoup\,{<}\,0\}.$ 
Then we can find 
a smooth $n{-}\qq{+}1$ dimensional smooth submanifold $\sfY$ of an open neighbourhood of
$\pct$ in 
$\sfX,$ with $\pct\,{\in}\,\sfY$ and
such that the restriction $\rhoup'$ of $\rhoup$ to $\sfY$ 
is strictly 
plurisubharmonic on $\sfY.$ 
 We can also assume that\footnote{If $Y_{1},Y_{2}$ are two smooth submanifolds
 of a manifold $\sfX$ and $\pct\,{\in}\,\sfX,$ the symbol $Y_{1}\,{\pitchfork}_{\pct}Y_{2}$
 means that they are \emph{transversal} at $\pct,$
i.e.  that  either
     $\pct\notin\sfY_{1}\cap{\sfY}_{2},$ or $\pct\in\sfY_{1}\cap{Y}_{2}$ 
     and $\mathrm{T}_{\pct}\sfY_{1}+\mathrm{T}_{\pct}\sfY_{2}=\mathrm{T}_{\pct}\sfX.$
     We write $\sfY_{1}\,{\pitchfork}\,{Y}_{2}$ if $\sfY{\pitchfork_{\pct}}Y_{2}$ for all $\pct\in\sfX.$}  
 $\sfY\,{\pitchfork}_{\,\pct}\,\sfA.$ 
 Transverality means that, if $\zq_1,\hdots,\zq_n$
are holomorphic coordinates at $\pct$ such that $\sfY{=}\{\zq_{1}{=}\,0,\hdots,\zq_{\qq{-}1}{=}\,0\}$ 
then $(\zq_1,\hdots,\zq_{\qq{-}1})$  define a regular sequence in $\cO_{\sfA,\pct}.$ 
Then $\sfA'=\sfA\cap\sfY$ is a complex curve in $\sfY$ and the fact that the subharmonic 
$\rhoup|_{\sfA'}$ has an isolated maximum at $\pct$ yields a contradiction (see e.g. \cite{scvVII}).
\end{proof}
Let  $\sfA$ be a locally closed
real smooth submanifold of a complex manifold $\sfX.$ Each point $\pct$ of $\sfA$
has an open neighbourhood $U_{\pct}$ in $\sfX$ such that 
\begin{equation}\label{eq-1.3}
 \sfA\cap{U}_{\pct}=\{\qct\in\sfX\mid \rhoup_{1}(\qct)=0,\;\hdots,\,\rhoup_{k}(\qct)=0\},
\end{equation}
with $\rhoup_{1},\hdots,\rhoup_{k}\,{\in}\,\Ci(U_{\pct},\R),$ $d\rhoup_{1}(\qct){\wedge}\cdots{\wedge}
d\rhoup_{k}(\qct)\,{\neq}\,0$ for $\qct\,{\in}\,U_{\pct}$. The integer $k$ is the real codimension of
$\sfA$ in $\sfX.$ The complex linear space 
\begin{equation*}
 \mathrm{T}^{1,0}_{\pct}\sfA=\{\wq\in\mathrm{T}^{1,0}_{\pct}\sfX\mid \partial\rhoup_{j}(\wq)=0,\; j=1,\hdots,k\}
\end{equation*}
(\emph{holomorphic tangent to $\sfA$ at $\pct$})
is independent of the choice of the locally defining functions $\rhoup_{1},\hdots,\rhoup_{k}.$ 
Its dimension, that we will indicate by $\crdim(\sfA,\pct),$  
is the \emph{$CR$-dimension} of $\sfA$ at~$\pct.$ The map $\pct\,{\mapsto}\,\crdim(\sfA,\pct)$
is upper semicontinuous on $\sfA$. We say that at the point $\pct$ the submanifold 
$\sfA$ is 
\begin{itemize}
 \item \emph{$CR$-regular} if $\crdim(\sfA,\,\cdot\,)$ is constant on
 a neighbourhood of $\pct$;
 \item  \emph{$CR$-singular}  otherwise.
\end{itemize}
Set, for $c\,{=}\,(c_{1},\hdots,c_{k})\,{\in}\,\R^{k},$ $\rhoup_{c}\,{=}\,{\sum}_{i=1}^{k}c_{i}\rhoup_{i}.$
For all $C\,{>}\,0,$ we have 
\begin{equation*}
 \sfA\cap{U}_{\pct}\subseteq\left\{\qct\in{U}_{\pct}\left| 
 \rhoup_{c}(\qct)+C{\sum}_{i=1}^{k}\rhoup_{i}^{2}(\qct)
 \leq{0}\right.\right\}.
\end{equation*}
\begin{lem}If the complex Hessian $i{\cdot}\partial\bar{\partial}\rhoup_{c}$ has $\ell$ positive eigenvalues
on $\mathrm{T}^{1,0}_{\pct}\sfA,$ then $\sfA$ is set-theoretically $([\crdim(\sfA,\pct)]{-}\ell{+}1)$-pseudoconvex
at $\pct.$ \end{lem} 
\begin{proof}
 Indeed, for sufficiently large $C\,{>}\,0,$ the number of positive eigenvalues of the
 complex Hessian of 
\begin{equation*}
 \tag{$*$}
 \qct\,{\to}\,\rhoup_{c}(\qct)+C{\sum}_{i=1}^{k}\rhoup_{i}^{2}(\qct)\end{equation*} 
 at $\pct$ 
is the sum of $\dim_{\C}\sfX{-}[\crdim(\sfA,\pct)]$  and of the number of its positive
eigenvalues on $\mathrm{T}^{1,0}_{\pct}\sfA.$ 
\end{proof}
Vice versa, a smooth $\rhoup$ satisfying the first line of \eqref{eq-1.2'} and vanishing at $\pct$
is bounded from below by a function of the form $(*)$. This yields 
\begin{prop}\label{p1.7}
Let  $\sfA$ be a locally closed
real smooth submanifold of a complex manifold $\sfX,$
locally defined by \eqref{eq-1.3} near its point $\pct.$  \par
Then $\sfA$ is set-theoretically
weakly $\qq$-peudoconcave at $\pct$ if and only if, for all $c\,{\in}\,\R^{k},$  
the complex Hessian $i{\cdot}\partial\bar{\partial}\rhoup_{c}(\pct)$  has
at least $\qq$ nonpositive eigenvalues on $\mathrm{T}^{1,0}_{\pct}\sfA$. \qed
\end{prop}
The fact that the complex Hessian $i{\cdot}\partial\bar{\partial}\rhoup_{c}(\pct)$  has
at least $\qq$ nonpositive eigenvalues on $\mathrm{T}^{1,0}_{\pct}\sfA$ is usually referred to
as \emph{weak $\qq$-pseudoconcavity}. \par
Thus Prop.\ref{p1.7} can be restated in the form \textquotedblleft{\textsl{For smooth real
submanifolds of complex manifolds the notions of weak $\qq$-pseudoconcavity and
set-theoretic weak $\qq$-pseudoconcavity coincide.}}\textquotedblright
\subsection{Intersections with complex submanifolds}
Slicing will be an important tool for studying properties of set-theoretically 
weakly pseudoconcave locally closed subsets. We prove here the
following Lemma.
\begin{lem}\label{l-1.8}
Let $\sfA$ be a  locally closed set-theoretically weakly 
$\qq$-pseu\-do\-con\-cave subset of a complex manifold $\sfX$
and $\sfY$ a complex submanifold of codimension $d$ in $\sfX.$
Then the intersection $\sfA\,{\cap}\,\sfY$ is set teoretically 
weakly 
$(\qq{-}d)$-pseudoconcave.
\end{lem} \begin{proof}
We argue by contradiction, assuming that $\sfA\,{\cap}\,\sfY$
is not set-theoretically weakly 
$(\qq{-}d)$-pseudoconcave at some point $\pct_0.$  
This means that 
there is a smooth real valued function $\rhoup,$
defined on an open neighourhood $\omegaup$ of $\pct_{0}$
in $\sfY,$ 
vanishing at $\pct_{0},$ 
whose complex hessian 
at $\pct_{0}$ 
has 
at least 
$$(n{-}d)\,{-}\,(\qq{-}d)\,{+}\,1 = n{-}\qq{+}1$$
positive eigenvalues and such that \begin{equation*}
\omegaup\cap\sfY\cap\sfA\cap\{\rhoup{\leq}0\}=\{\pct_{0}\}.
\end{equation*}
\par 
We can as well assume that $\sfA$ is a closed subset of 
the euclidean ball
$\sfX\,{=}\,\{\zq\,{\in}\,\C^n\,{\mid}\,|\zq|\,{<}\,1\},$
that 
the point $\pct_{0}$ is $0$ and 
that 
$\sfY$ is the complex 
$(n{-}d)$-plane $\{\zq_i\,{=}\,0\,{\mid}\, 1{\leq}i{\leq}d\}.$ 
The function $\rhoup$ lifts to a smooth real valued function,
independent of $\zq_1,\hdots,\zq_d,$ 
defined on the cylinder $\C^{d}{\times}\,\omegaup$
in $\C^n,$ 
whose complex hessian has at least $(n{-}\qq{+}1)$ positive
eigenvalues on all $(n{-}d)$-planes parallel to $\sfY.$ 
For real $C{>}0$ we
consider the functions \begin{equation*}
\rhoup_{C}(\zq)=\rhoup(\zq)-C{\sum}_{j=1}^d|\zq_i|^2.
\end{equation*}\par 
Let $0\,{<}\,r\,{<}\,1$ be so small that $\{|\zq|{\leq}r\}\cap\sfY\}
\Subset\omegaup.$ Since $\rhoup$ is strictly negative on
$\sfA{\cap}\sfY{\cap}\{|\zq|{=}r\},$ 
it stays negative on a neighourhood and  
we can find $\epsilonup_1,\epsilonup_2\,{>}
\,0$ such that 
$\rhoup_{C}(\zq){<}\,{-}\epsilonup_1$ if $|\zq|{=}r$
and ${\sum}_{j=1}^d|\zq_j|^2{\leq}\epsilonup_2^2.$ 
As $\rhoup$ is bounded on $\{|\zq|{=}r\},$ 
by taking $C\,{>}\,0$ sufficiently large we can obtain that
$\rhoup_C(\zq){<}{-}\epsilonup_1{<}0$ on  $\sfA\,{\cap}\,\{|\zq|{=}r\}.$ 
Because $\rhoup_{C}(0){=}0,$ 
the restriction of 
$\rhoup_{C}$ to $\sfA{\cap}\{|\zq|{\leq}r\}$
has a nonnegative maximum $c_0$
at some point with $|\zq|{<}r.$
Since the complex hessian of the
function $\rhoup_{C}$ has at least 
$n{-}\qq{+}1$ positive eigenvalues at all points of
$\{|\zq|{\leq}r\},$ the inclusion $A{\cap}\{|\zq|{<}r\}\,{\subseteq}
\,
\{\rhoup_{C}\,{\leq}\,{c}_0\}$ yields a contradiction, proving
our statement. 
\end{proof}
Analogous slicing results were known in the case of embedded pseudoconcave $CR$ manifolds,
under the additional assumption that $\sfY$ is \textit{transverse} to the the holomorphic tangent space 
$H_\pct M$ at each $\pct\in M\cap\sfY,$
ensuring that $M\cap\sfY$ is still a $CR$ manifold.

\begin{exam}
\label{e-1.9} 
Lemma~\ref{l-1.8} 
may be used to construct 
submanifolds which are set-theoret\-i\-cally
pseudoconcave and $CR$-singular at some point. 
We start from the rigid $n$-pseu\-do\-con\-cave $CR$ submanifold
\begin{equation*}
M \coloneqq \begin{cases}
\uq^1{+}\bar{\uq}^1=|\zq|^2-|\wq|^2,\\
\uq^2{+}\bar{\uq}^2
={\sum}_{i=1}^{n}(\zq^i{\cdot}\bar{\wq}^i+\bar{\zq}^i{\cdot}\wq^i),\\
\uq=(\uq^{1},\uq^{2})\in\C^2,\; \zq,\wq\in\C^n
\end{cases} 
\end{equation*}
of $\C^{2n+2}\!{=}
\C^{2}_{\uq}{\times}\,\C^{n}_{z}{\times}\,\C^{n}_{w}.$ \par
For real 
$k\,{>}\,0$ the complex hypersurface $\sfY_k{=}\,\{k\uq^1\,{=}\,\iq\,\uq^2\}$
is transversal, but not $CR$-transversal to $M$ at $0.$
Hence the intersection $M_k=M\,{\cap}\,\sfY_k$ is a smooth submanifold,
which is not $CR$ at $0,$ but is still set-theoretically weakly
$(n{-}1)$-pseudoconcave. \par
Taking holomorphic coordinates $\uq^{1},\zq,\wq$ on $\sfY_{k}$, we get 
\begin{equation*}
 \sfM_{k}:\; 
\begin{cases}
 \uq^1{+}\bar{\uq}^1=|\zq|^2-|\wq|^2,\\
 \uq^{1}{-}\bar{\uq}^{1}=\tfrac{i}{k}(\wq^{*}\zq{+}\zq^{*}\wq).
\end{cases}
\end{equation*}
The two differentials 
\begin{equation*} 
\begin{cases}
 du^{1}-\zq^{*}d\zq+\wq^{*}d\wq,\\
 du^{1}-\tfrac{i}{k}\wq^{*}d\zq-\tfrac{i}{k}\zq^{*}d\wq,
\end{cases}
\end{equation*}
are linearly dependent at points where 
\begin{equation*} \begin{cases}
 \zq={-}\tfrac{i}{k}\wq,\\
 \wq=\tfrac{i}{k}\zq.
 \end{cases}
\end{equation*}
This shows that $0$ is an isolated $CR$ singularity for $k\,{\neq}\,1,$ while for $k\,{=}\,1$
the locus of $CR$ singularities is the complex $n$-plane
$\{\wq\,{=}\,i\zq,\; \uq^{1}{=}0,\,\uq^{2}{=}0\}.$ 

\par
To study the stability of these $CR$ singularities,
we consider the Gauss map $\gammaup\,{:}\,\pct\,{\mapsto}\,\mathrm{T}_{\pct} \sfM_k,$
thought, by using the affine structure of $\sfY_{k},$  
as taking values in the Grassmannian  
$\Gr^{\R}_{4n}(\sfY_{k})$ 
of real $4n$-dimensional
subspaces of~$\sfY_k$.
Being $\sfY_{k}\,{\simeq}\,\R^{4n+2}$,  
the Grassmannian $\Gr^{\R}_{4n}(\sfY_{k})$ 
has 
real 
dimension
$8n$ and contains the grassmannian 
$\Gr_{2n}(\sfY_{k})$ of its \textit{complex} $2n$-dimensional linear subspaces 
as a submanifold of real dimension $4n.$\par 
We claim that $\gamma$ is transverse to 
$\Gr_{2n}(\sfY_{k})$ 
at the point $\{u\,{=}\,0\}$ if $k\,{\neq}\,1.$
In these cases, transversality implies that the CR singularity
cannot be removed by small deformations of $\sfM_k.$ 

To check the claim, we use Garrity's transversality criterion
\cite{Gar}, (see also \cite{Coff}).
It relies on the fact that every $2$-codimensional submanifold 
tangent to the $(z,w)$-space at $0$ can be represented as
\[
u = (\bar{z}^\intercal \;\bar{w}^{\intercal})\, R 
\left(\!\!\begin{array}{c}  z \\ w \end{array}\!\!\right)
+\re\left((z^{\intercal}\; w^{\intercal})\, P 
\left(\!\!\begin{array}{c}  z \\ w \end{array}\!\!\right)\right)
 +O(|(z,w)|^3),
\]
with complex $2n{\times} 2n$ matrices $R$ and $P.$
In general, this may require a further quadratic coordinate change,
but in our simpler case we directly get 
$R=
\left(\!\!\begin{array}{cc}  I_{2n} & ikI_{2n} \\ ikI_{2n} & -I_{2n} \end{array}\!\!\right),$
$P=0,$ with zero error term.
By Garrity's criterion, transversality at $0$ is
equivalent to the fact that
$R$ has nonzero determinant 
and one easily verifies that this happens when $k\not=1.$
We will come back to CR singularities in \S{\ref{S7}}.
\end{exam}

\section{Locally approximable generalized $CR$ functions}
\subsection{Closure operations on sheaves of continuous functions} \label{sub1-1}
Let $\sfA$ be a first-countable Hausdorff space. 
Denote by $\cC_{\!\sfA}$
 the sheaf of germs of continuous
complex valued functions on $\sfA,$  by $\Cg_{\!\sfA}$ 
its total space, spread over $\sfA,$ 
and let $\cH$ be any subset of $\Cg_{\!\sfA}$.
For $U^{\text{open}}\,{\subseteq}\,\sfA$
we use the notation $\cH(U)$ for the set of sections $f{\in}\cC(U)$ such that, 
for each $x\in{U}$, the germ $f_{(x)}$ belongs to $\cH$.  
We define the \emph{local uniform extension} $\cH^c$ 
of 
$\cH$ 
to be the set 
of all germs $f_{(x)}$ having, for some neighbourhood  $U_x$ of $x$ in $\sfA,$ 
a representative $f{\in}\cC_{\sfA}(U_x),$ which is on $U_x$ 
the uniform
limit 
of a sequence $\{f_{\nuup}\}\,{\subset}\,\cH(U_{x})$
on compact subsets of $U_{x}$.
The correspondence $\cH\to\cH^c$
is a \emph{closure operation} in the sense of \cite{Ric}: this means that the following are satisfied: 
\begin{align}
 \tag{$i$} & \emptyset^c=\emptyset,\\
 \tag{$ii$} & \cH\subseteq\cH^c,\;\;\forall \cH\subseteq\Cg_{\!\sfA};\\
 \tag{$iii$} & \cH_1\subseteq\cH_2\subseteq\Cg_{\!\sfA}
 \Longrightarrow \cH_1^c\subseteq\cH_2^c.
\end{align}\par
When $U\,{\to}\,\cH(U)$ is a sheaf, 
its uniform extension $U\,{\to}\,\cH^c(U)$ is also a sheaf.
\par
We say that $\cH$ is \emph{locally uniformly closed} if $\cH\,{=}\,\cH^c$.\par
The \emph{local uniform closure} $\cH^{(c)}$
of $\cH$  is the smallest
locally uniformly closed subset of $\Cg_{\sfA}$ containing $\cH$.
In fact 
$\cH^{(c)}$ may be larger than $\cH^c$
(see Example~\ref{ex1-5} below).
The following is 
a basic property of this construction, (see 
\cite[Prop.8.1]{Ric}):
By employing transfinite induction, 
one can uniquely associate 
to every subsheaf $\cH$ of $\cC_{\sfA}$ an ordinal $\mu$ and
a \textit{resolution} $\{\cH_\nu\,{\mid}\,\nu\,{\preceq}\,\mu\}$ 
consisting of
a well ordered family of sheaves,
such that 
\begin{align}
\tag{$i$} & \cH_0=\cH, \;\;\cH_\mu=\cH^{(c)}, \\
\tag{$ii$} & \cH_\alpha\subsetneqq \cH_\beta
\;\;\text{if $0\preceq\alpha\prec\beta\preceq\mu$,}\\
\tag{$iii$} &\cH_\nu
=\left({\bigcup}_{\alpha\prec\nu} \cH_\alpha\right)^c 
\;\;\text{for $\nu\preceq\mu$.}
\end{align} \par
It is convenient to define $\cH_{\alpha}$ also for 
larger ordinals,
by setting $\cH_{\alpha}\,{=}\,\cH_{\mu}$ for~$\alpha\,{\succ}\,\mu.$ \par 
If $\sfA$ is a $\sigmaup$-locally compact\footnote{A topological space is said to be 
$\sigmaup$-compact if it is the union of countably many compact subspaces
and $\sigmaup$-locally compact if it is both $\sigmaup$-compact and locally compact.}
Hausdorff
space, 
then for each open $U$ in $\sfA$ the space
$\cC_{\!\sfA}(U)$ of continuous complex valued functions on $U$
is Fr\'echet for the topology of uniform convergence on
compact subsets of $U.$ 
We recall that a sheaf $\cS$
of vector subspaces
 of $\cC_{\!\sfA}$  
 is \textit{Fr\'echet} if 
we can find a basis $\cU\,{=}\,\{U_{i}\}$ 
of open sets in $\sfA$ such that $\cS(U_i)$ is a
Fr\'echet subspace of $\cC_{\!\sfA}(U_{i})$ for all
indices $i.$ 
\begin{lem} If $\sfA$ is a $\sigmaup$-locally compact 
Hausdorff
space, then the uniform closure $\cH^{(c)}$ of 
any  
subsheaf of vector spaces of 
$\cC_{\!\sfA}$  is a Fr\'echet sheaf.\qed
\end{lem}

\subsection{Locally approximable $CR$ functions}
Let $\sfX$ be a complex manifold.
Denote by
$\cO_{\sfX}$ the sheaf of germs of holomorphic functions on $\sfX,$ by $\Og_\sfX$ its total space and
by $\piup_{\sfX}:\Og_{\sfX}\,{\to}\,\sfX$ its spreading over $\sfX.$
We apply the notions of 
local uniform extension and closure 
introduced in \S\ref{sub1-1} 
to define continuous $CR$ functions on rather general subsets of 
$\sfX.$ 
\par \smallskip
Let $\sfA$ be a  locally closed subset 
of $\sfX$ 
and 
$\cO_{\sfX}|_{\sfA}$ the  induced sheaf
on $\sfA$ by the inclusion map $\sfA\,{\hookrightarrow}\,\sfX$
(cf. \cite{gode}). By evaluating at points of $\sfA$, we obtain a
morphism of $\cO_{\sfX}|_{\sfA}$ into a subsheaf 
of vector spaces of $\cC_{\sfA}$. 
\begin{defn}\label{defnstcr}
We define the sheaf $\tiO_{\sfA}$ 
of germs of 
\emph{locally approximable continuous $CR$ functions}
(in the sequel, for short, $LACR$) 
to be
the local uniform closure 
$\cO_{\sfA}^{(c)}$ of $\cO_{\sfA}$.
\end{defn}
We have (see e.g. \cite{Andreotti1973}) 
\begin{lem} \label{l-2.2}
For any  locally closed $\sfA\subseteq\sfX,$ 
$\tiO_{\sfA}$ is a Fr\'echet sheaf. \qed
\end{lem}
We will indicate by 
$\cO_{\sfA,\alphaup}$ the terms in the resolution
of \S\ref{sub1-1} with first term $\cO_{\sfA,0}\,{=}\,\cO_{\sfA}$ and use {\textquotedblleft{$\, |_{\sfA}$
\textquotedblright} to signify 
{\textquotedblleft{restriction to $\sfA$}\textquotedblright},
e.g.  of a continuous function defined on a 
subset $E$ of the ambient space $\sfX$
to the intersection~$E\,{\cap}\,\sfA$. 
\par\smallskip 
LACR functions are well behaved with respect to restrictions.
\begin{lem}\label{l-2.3}
If $\sfA,\sfB$ are locally closed subsets of 
 a reduced complex space $\sfX,$ with $\sfA\,{\subseteq}\,\sfB$,
 then $\left.(\cO_{\sfB})_{\alphaup}\right|_{\sfA}\,{\subseteq}\,(\cO_{\sfA})_{\alphaup}$
 for all ordinals $\alphaup$. \par
 In particular, restriction of functions yields a morphism
 $\rtt^{\sfB}_{\sfA}:\tiO_{\sfB}\,{\to}\,\tiO_{\sfA}$, which 
is 
continuous 
with respect to the natural
 Fr\'echet structures.
\end{lem} 
\begin{proof} We prove the statement by transfinite induction. 
It is clear in fact that all germs of $\cO_{\sfA}$ at a point $\pct$ of $\sfA$ are
restrictions of germs of ${\cO_{\sfB}}_{(\pct)}$.
\par
Let $\alphaup{\succ}0$ and assume that 
$\left.\cO_{\sfB,\betaup}\right|_{\sfA}\,{\subseteq}\, \cO_{\sfA,\betaup}$ for $\betaup{\prec}\alphaup.$ 
 If $f\,{\in}\,\cO_{\sfB,\alphaup}(U\,{\cap}\,\sfB)$ for an $U$ open in $\sfX,$ then for each
 point $\pct$ of $U\,{\cap}\,\sfA$ we can find an open neighbourhood $U'$ of $\pct$ in $\sfX$
 and sequences $\{\betaup_j\}$ of ordinals smaller than ${\alphaup}$,  
 $\{f_{j}\}$ with $f_{j}\,{\in}\,\cO_{\sfB,\betaup_{j}}(U'{\cap}\sfB)$, 
such that  
$f_{j}{\to}f$ uniformly on compact subsets of $U'{\cap}\sfB.$
 By the inductive assumption, $\left.f_{j}\right|_{U'\cap\sfA}\,{\in}\, \cO_{\sfA,\betaup_{j}}(U'{\cap}\sfA)$
 and, since $\left.f_{j}\right|_{U'\cap\sfA}$ converges to $\left.f\right|_{U'\cap\sfA}$ uniformly on compact
 subsets of $U'{\cap}\sfA,$ this means that 
 $\left. f\right|_{U'\cap\sfA}\,{\in}\,\cO_{\sfA,\alphaup}(U'{\cap}\sfA).$
 This is true for all points $\pct$ of $\sfA$, 
 proving that \hfill\linebreak
 $\left.\cO_{\sfB,\alphaup}\right|_{\sfA}\,{\subseteq}\,\cO_{\sfA,\alphaup}.$ 
 Since this is true for all ordinals $\alphaup,$ it follows that
 $\left.\tiO_{\sfB}\right|_{\sfA}\,{\subseteq}\,\tiO_{\sfA}.$
 Since $\tiO_{\sfB}$ and $\tiO_{\sfA}$ are Fr\'echet subsheaves of $\cC_{\sfB}$ and $\cC_{\sfA},$
 the restriction map is still continuous for their Fr\'echet structures.
\end{proof}
\begin{exam}\label{ex2.4}
If $\sfA$ is a reduced complex analytic subvariety 
of a complex manifold $\sfX,$ then $\cO_{\sfA}$ is a
Fr\'echet sheaf (see e.g. \cite[Thm.11.2.5]{Taylor2002})
and hence 
$\tiO_{\sfA}{=}\cO_{\sfA}$.
\end{exam}

\begin{exam}
If $\sfA$ is a 
smooth $CR$ submanifold of a complex manifold $\sfX$, 
the Baouendi-Treves approximation theorem 
\cite{ba-tr} implies that $LACR$ functions 
are the same as usual continuous $CR$ functions 
and $\tiO_{\sfA}=\cO_{\sfA,1}.$ 
\end{exam}
The smallest ordinal $\muup$ for which $\tiO_{\sfA}\,{=}\,\cO_{\sfA,\muup}$ 
may be larger than $1$.

\begin{exam}
\label{ex1-5}
Let $\sfA$ be a smooth real surface properly embedded in a $2$-di\-men\-sional 
complex manifold $\sfX.$
The \emph{complex} points of $\sfA$ are those $\pct$ for which $\mathrm{T}_{\!\pct}\sfA$ is a complex
line. In suitable holomorphic coordinates centered at its 
isolated 
complex point $\pct$
the real surface 
$\sfA$ can be represented 
by
\begin{equation} \label{eq-2.1}
z_2=|z_1|^2+\gammaup(z_1^2+\bar{z}_1^2)+O(|z_1|^3)=(1+2\gammaup)x_1^2+(1-2\gammaup)y_1^2+
O(|z_1|^3),
\end{equation}
for a real $\gammaup\,{\geq}\,0$, which is a local biholomorphic invariant.
Following \cite{B65}, $\pct$ is called an \textit{elliptic}, \textit{parabolic} or 
\textit{hyperbolic} point 
if $\gammaup{<}1/2$, $\gammaup{=}1/2$ or $\gammaup{>}1/2$, respectively. 
Elliptic and hyperbolic complex points are isolated 
and stable with respect to small deformations of $\sfA.$
\par
Let us assume 
next 
that $\sfA$ is \textit{in generic position}, 
meaning that its complex points are isolated and either elliptic or hyperbolic. 
It is known that on~$\sfA$
each totally real point  
has a 
fundamental system of open neighbourhoods on which
every continuous function can be 
uniformly approximated
by restrictions of holomorphic functions. 
By \cite{FS91} the same holds at hyperbolic complex points.
This implies $\tiO_{\sfA}\,{=}\,(\cO_{\sfA})_1=\cC_{\!\sfA}$ if $\sfA$ has only hyperbolic complex points.
\par 
To examine the local geometry at an elliptic complex point $\pct,$
we can assume that $\sfX$ is a ball around 
$\pct{=}0$ in $\C^2.$
In \cite{B65} Bishop shows that 
there is $s_{0}{>}0$ and 
a family of complex discs
$\Delta_s,$ $0{\leq} s{\leq} s_0,$ attached to $\sfA$ at $0$.
More precisely, $\Delta_s:\bar{\D}\rightarrow\C^2$ is continuous
on the closed unit disc $\bar{\D},$  holomorphic in its interior 
and maps the unit circle into $\sfA.$ 
Moreover, it may be arranged that $\Delta_s$ is continuous 
on $\bar{\D}\times[0,s_0]$, injective on $\bar{\D}\times(0,s_0]$,
while $\Delta_0\equiv 0.$
\par 
If $f$ is the uniform limit on a neighbourhood $V$ of $0$ in $\sfA$
of a sequence of functions that are holomorphic on some open neighbourhood 
of $V$ in $\C^{2},$ 
then
the maximum modulus principle implies that, 
for $s$ sufficiently small, 
$f\circ \Delta_s(\zeta),$ $|\zeta|=1,$ has a continuous extension
which is holomorphic on $\D.$ 
Thus $(\cO_{\sfA})_1$ is a proper subsheaf of $\cC_M.$
On the other hand, if we fix a continuous function $\chiup$ on $\C^{2},$ which equals
$0$ for $|z|{<}(1/2)$ and  $1$ for $|z|\,{>}\,1,$ then 
for a continuous $g\,{\in}\,\cC_M(V)$ and any positive integer $\nuup,$ 
the function
$g_{\nuup}(z)\,{=}\,g(0){+}\chiup(\nuup{\cdot}z)(g(z){-}g(0))$ belongs to
$(\cO_{\sfA})_{1}(V),$ being constant on a neighbourhood of $0,$ 
and the sequence $\{g_{\nuup}\}$ converges uniformly to $g$ on $V$. 
Therefore  $(\cO_{\sfA})_{0}\,{\subsetneqq}\,(\cO_{\sfA})_{1}
\,{\subsetneqq}\,(\cO_{\sfA})_{2}\,{=}\,\cC_{\sfA}$ 
when $\sfA$ has elliptic complex points.
 \par
Finally we observe that $\sfA$ is set-theoretically $1$-pseudoconvex at every $\pct\,{\in}\,\sfA.$
This is well-known for $\pct$ totally real. Otherwise we rewrite
\eqref{eq-2.1} as 
\[
z_2=|z_1|^2+\gammaup(z_1^2+\bar{z}_1^2)+g(z_1,\bar{z}_1)+ih(z_1,\bar{z}_1)
\]
with real-valued third-order terms $g$ and $h.$
The hypersurface
$$\{x_2=|z_1|^2+\gammaup(z_1^{2}{+}\bar{z}_1^{2})+g(z_1,\bar{z}_1)\}$$
(which contains $\sfA$) is strictly pseudoconvex at $0.$
\end{exam}
\begin{exam}
\label{ex2-7} When $k$ is a positive integer, 
\begin{equation*}
\sfA= \sfA_k=\{z\in\C^3 \mid x_3=(|z_1|^2+|z_2|^2)^{k},y_3=0\}
\end{equation*}
is a $\Ci$-smooth real $4$-dimensional submanifold of $\C^{3},$ 
having an isolated $CR$ singularity at $0$ 
and is a generic $CR$ manifold elsewhere. 
We consider on $\sfA$ a \textit{gap sheaf} 
$\tiO_{\sfA,[0]}$, defined by
\begin{equation*}
{\tiO}_{\sfA,[0]}(U)=\tiO_{\sfA}(U\backslash\{0\})\cap\cC_{\!\sfA}(U).
\end{equation*}
Set
$U_\epsilon\,{=}\,\sfA\cap\{x_3\,{<}\,\epsilon\},$ $H\,{=}\,\{y_3\,{=}\,0\},$ 
$H_\epsilon\,{=}\,\{(|z_1|^2+|z_2|^2)^{k}\,{<}\,x_3\,{<}\,\epsilon,\,y_3\,{=}\,0\}$. 
Slice-wise applying the Hartogs-Bochner theorem yields unique extension from
$\tiO_{\sfA,[0]}(U_\epsilon)$ to $\tiO_H(H_\epsilon)\cap\cC(H_\epsilon\cup U_\epsilon).$
Since ${\tiO}_H\,{=}\,(\cO_H)_1$ and since restrictions of $(\cO_H)_1$ functions are of class $(\cO_{\!\sfA})_1,$  
approximation of $LACR$ functions on $\sfA$ by the values of their extensions on
$\{x_3{=}(|z_1|^2{+}|z_2|^2)^{k}{+}\eta,\,y_3{=}0\}$, where $\eta\,{\downarrow}\, 0,$ 
shows that ${\tiO}_{\!\sfA}=(\cO_{\!\sfA})_2=\tiO_{\sfA,[0]}.$  
\par
For $k\,{=}\,2,$ higher 
regularity is not preserved by CR extension.
As observed in \cite{LNR} the restriction $\sqrt{x_3}\,|_{\sfA}$ is of class $\cC^\infty\!,$
whereas its extension is only H\"older continuous at $0.$ 
\end{exam}
It may be interesting to study 
in general 
when $\tiO_{\sfA,[\pct]}$ equals ${\tiO}_{\!\sfA}.$
In  \S\ref{sect6}
we will verify this for a class of manifolds generalizing Ex.\ref{ex2-7}.
One may wonder
whether set-theoretical weak $1$-pseu\-do\-con\-cavity is enough, 
which would amount to a generalization 
of the Riemann removability theorem. 
However the following simple example shows that at least 
some {\textquotedblleft{normality condition}\textquotedblright}\, 
has to be added.
\begin{exam}
 Let us consider the complex analytic subvariety 
\begin{equation*}
 \sfA=\{z^2=w^3\}\subset\C^{2},
\end{equation*}
which is set-theoretically weakly $1$-pseudoconcave by Corollary \ref{cor1.4}.
In this case the gap sheaf $\tiO_{\sfA,[0]}$ is the normalization of the sheaf $\cO_{\sfA}$,
which equals $\tiO_{\sfA}$ (see Ex.\ref{ex2.4}). The two are different because $z/w$ 
(extended by $0$ at $0$)
defines
a global section of $\tiO_{\sfA,[0]}$, whose germ at $0$ does not belong to~$\cO_{\sfA}$.
\end{exam}
\begin{rmk} Let $\sfA$ be a closed subset of a smooth manifold $\sfX$. 
Let $\cE_{X}$ be the sheaf of germs of smooth complex valued functions on $\sfX$
and 
$\Zz_{\sfX,\sfA}$ its subsheaf consisting of germs which 
vanish to infinite order 
on $\sfA$.
The sheaf $\cW_{\!\sfX,\sfA}$ 
of germs of 
Whitney functions (of infinite order)
on $\sfA$  is the quotient sheaf 
defined by the exact sequence 
\begin{equation*} 
\begin{CD}
 0 @>>> \Zz_{{\sfX,\sfA}}@>>>\cE_{\sfX}@>>>\cW_{\!\sfX,\sfA}@>>>0.
\end{CD}
\end{equation*}
Note that, if $\pct_{0}\,{\in}\,\sfX{\setminus}\sfA,$ then $\cW_{\sfX,\sfA,\,\pct_{0}}\,{=}\,\{0\}$.
In general, if we fix any smooth local chart on an open $U\,{\subset}\,\sfX$, then
a section $f$ of $\cW_{\sfX,\sfA}$ on $U$ uniquely determines the Taylor series of its representative
in $\cE_{\sfX}(U)$ at points of $\sfA\,{\cap}\,U$. The coefficients of these Taylor series need to
satisfy extra conditions: e.g. 
when $\sfA$ is a smooth real submanifold of codimension $k$ in $\sfX$, $\pct_{0}\,{\in}\,\sfA,$
and, for an open neighbourhood $U$ of $\pct_{0}$ in $\sfX,$  
\begin{equation*}
 \sfA\cap{U}=\{\pct\in{U}\mid \rhoup_{1}(\pct)=0,\;\hdots,\,\rhoup_{k}(\pct)\,{=}\,0\},
\end{equation*}
with 
\begin{equation*}
 d\rhoup_{1}(\pct)\wedge\cdots\wedge{d}\rhoup_{k}(\pct)\neq{0}\,\;\;\forall\pct\in\sfA\cap{U},
\end{equation*}
then germs of $\cW_{\sfX,\sfA}$ at $\pct_{0}$ are represented by  
formal power series 
\begin{equation*}\tag{$*$} f=
 {\sum}_{i_{1},\hdots,i_{k}=0}^{\infty} f_{i_{1},\hdots,i_{k}}\rhoup_{1}^{i_{1}}\cdots\rhoup_{k}^{i_{k}}
\end{equation*}
in the transversal variables $\rhoup_{1},\hdots,\rhoup_{k},$ whose
coefficients $f_{i_{1},\hdots,i_{k}}$ are smooth functions defined on a common
relatively open neighbourhood $\omegaup$ of $\pct_{0}$ in~$\sfA$.
We note that from $(*)$ we can compute the Taylor series of $f$ at $\pct_{0}$ for any choice
of smooth coordinates on $\sfX$ about $\pct_{0}$. \par
When the ambient $\sfX$ is a complex manifold
it is natural (cf. e.g. \cite{N85}) to define germs of \emph{Whitney $CR$ functions} 
(WCR for short) on $\sfA$ 
by the requirement
that their Taylor series in holomorphic coordinates are formal series of holomorphic polynomials.
When $\sfA$ is a smooth generic $CR$ submanifold of $\sfX$, the notion of  
WCR
function coincides, by the equivalence of {\textquotedblleft{\textit{generic}}\textquotedblright} 
and  {\textquotedblleft{\textit{formally noncharacteristic for $\bar{\partial}\,$}}\textquotedblright} 
(see e.g.  \cite{AH1972a}),
with that of smooth $CR$ function.\par
The last observation in Ex.\ref{ex2-7} shows that instead, when a smooth $\sfA$ carries
$CR$ singularities,
smooth LACR functions on $\sfA$ may not 
be WCR
(see also \cite[Thm.2.4]{harris1979}). 
\end{rmk}
\par\medskip

\section{Maximum modulus principle $\mathrm{I}$}\label{sec-max}
Example \ref{ex1-5} shows that  
the sheaf of $LACR$
functions may 
coincide with the sheaf of germs of 
complex valued continuous functions even in cases where
the objects we consider  are not totally real. 
Our purpose is to show that Conditions \eqref{eq-1.2}
in Def.\ref{defn1.1a}
are appropriate to characterize classes of locally
 closed subsets
of reduced analytic spaces whose $LACR$ functions 
have special properties. We begin by investigating
relationships between 
weak pseudoconcavity and  the validity of  
\emph{local}
maximum modulus principles. 
\begin{prop}\label{p3.1}
Let $\sfA$ be a locally closed subset of 
a reduced complex space $\sfX.$ 
The following are equivalent: 
\begin{enumerate}
 \item $\sfA$ is 
set-theoretically weakly one-pseudoconcave;
\item 
for every $\pct\,{\in}\,\sfA$ there is an open
neighbourhood $U_\pct$ of $\pct$ in $\sfA$ such that 
\begin{equation}\label{eq-3-1a}
\forall V^\text{open}\Subset{U}_\pct ,\;\;
\forall u\in\cO_{\sfA}(V)\cap\cC^0(\bar{V}),\;\;
{\sup}_V|u|\leq{\sup}_{\partial_{\sfA}{V}}|u|.
\end{equation}
\item 
there is an ordinal $\alphaup\,{\succ}\,0$ such that
for every $\pct\,{\in}\,\sfA$ there is an open
neighbourhood $U_\pct$ of $\pct$ in $\sfA$ such that 
\begin{equation}\label{eq-3-1b}
\forall V^\text{open}\Subset{U}_\pct ,\;\;
\forall u\in\cO_{\sfA,\alphaup}(V)\cap\cC^0(\bar{V}),\;\;
{\sup}_V|u|\leq{\sup}_{\partial_{\sfA}{V}}|u|.
\end{equation}
\item 
for every $\pct\,{\in}\,\sfA$ there is an open
neighbourhood $U_\pct$ of $\pct$ in $\sfA$ such that 
\begin{equation}\label{eq-3-1}
\forall V^\text{open}\Subset{U}_\pct ,\;\;
\forall u\in\so_{\sfA}(V)\cap\cC^0(\bar{V}),\;\;
{\sup}_V|u|\leq{\sup}_{\partial_{\sfA}{V}}|u|.
\end{equation}
\end{enumerate}

\end{prop} 
In fomulas \eqref{eq-3-1a}, \eqref{eq-3-1b}, \eqref{eq-3-1} 
we wrote ${\partial_{\sfA}{V}}$
to stress the fact that we are taking 
the boundary of $V$ 
with respect to 
the subspace topology of $\sfA.$ 
\begin{proof} As $\cO_{\sfA}\,{\subseteq}\,
\cO_{\sfA,\alphaup}\,{\subseteq}\,\so_{\sfA},$  we have $(4)\,{\Rightarrow}\,(3)\,{\Rightarrow}\,(2)$. 
\par
We prove that $(2)\,{\Rightarrow}\,(1)$ by contradiction. 
If $\sfA$ is set-theoretically $1$-pseu\-do\-convex
at $\pct,$ then there is a local model of $\sfX$ at $\pct$ representing
a neighbourhood $\omegaup_{\pct}$ of $\pct$ in $\sfA$ as a closed
subset of an open subset $U$ of $\C^{N}$ for which there is 
an $f\,{\in}\,\cO(U)$ with $\re(f(\pct))\,{=}\,0$ and $\re(f(\qct))\,{<}\,0$
for $\qct\,{\in}\,\omegaup_{\pct}\,{\backslash}\{\pct\}.$
Clearly $u\,{=}\,\exp(f)$ violates 
the local maximum modulus principle \eqref{eq-3-1a}
for $\pct\,{\in}\,V^{\text{open}}\,{\Subset}\,\omegaup_{\pct}.$ 
Therefore
the local maximum modulus principle \eqref{eq-3-1a}
is a necessary condition for set-theoretic weak $1$-pseudoconcavity.
\par  
Let us prove that $(1)\,{\Rightarrow}\,(4)$.
We will do it in two steps, showing first that $(2)$
is valid 
and then extending this result 
by transfinite induction.\par 
While proving the local maximum modulus principle
we can restrain to the case where
$\sfA$ is a closed subset of the open ball $\{|\zq|{<}1\}$  
of 
$\C^n.$ We show that in this case \eqref{eq-3-1a}
is true 
for all relatively compact open subsets of $\sfA.$ 
Assume by contradiction that there are 
an open subset $U$ of $\sfA$ and 
a function
$f\in\cO_{\!\sfA}(U)\cap\cC^0(\bar{U})$  
for which \eqref{eq-3-1a} is false. 
We can assume that $0\,{\in}\,U$ and that
$1
{=}\max_{\overline{U}}|f|{>}\max_{\partial_{\sfA}{U}}|f|.$
Let us consider the embedding \begin{equation*}
\phiup:\bar{U} \ni \zq \longrightarrow (\zq, f(\zq))\in\C^{n+1}
=\C^{n}_{\zq}\times\C^1_{\wq}.
\end{equation*}
For $\epsilon,c\in\R$ and 
$\epsilon>0,$
the domains 
$\Omega_{\epsilon,c}=\{\re(w)+\epsilon|z|^2<c\}$
are strictly pseudoconvex. Moreover, if $\epsilonup>0$ 
is sufficiently
small, the set  
$\phiup(\partial_{\sfA}{U})$ 
is contained
in $\Omega_{\epsilon,c}$ for all $c\geq{1}.$ Having fixed an
$\epsilon>0$ with this property,  
there is a minimal 
$c\geq{1}$ 
for which $\phiup(U)\subset
\overline{\Omega}_{\,\epsilon,c}.$
In particular, we can find 
$z_1\,{\in}\,U$ with
$(z_1,f(z_1))\in\partial\Omega_{\epsilon,c}.$ 
Then we take a $0<\delta<\epsilon,$ in such a way that 
\begin{equation*}
\Omega'=\{z\in\C^{n+1}\mid \re(w)+\epsilon |z|^2-
\delta |z-z_1|^2<c\}
\end{equation*}
has
still a strictly pseudoconvex boundary. 
We fix a holomorphic extension $\tilde{f}$ of $f$ on 
a neighbourhood $\omega$ of $z_1$ in $\sfX$ and, accordingly,
we extend $\phiup$ to a holomorphic embedding
$\tilde{\phiup}$ of $\omega$ into
$\C^{n+1}.$  
We have
$\tilde{\phiup}(\sfA\,{\cap}\,{\omega}{\setminus}\{z_1\})\,{\subset}\,\Omega'$ and
$\tilde{\phiup}(z_1)\,{\in}\,\partial\Omega',$ contradicting the 
assumption that $\sfA$ is set-theoretically $1$-pseudoconcave 
at~$z_1.$ 
This establishes $(2)$. 
\par \smallskip
Let $\{\cO_{\sfA,\alphaup}\,{\mid}\,\alphaup{\preceq}\muup\}$ be the 
well-ordered family of sheaves,
with $\cO_{\sfA,0}{=}\cO_{\sfA}$ 
and $\cO_{\sfA,\muup}{=}\tiO_{\sfA},$ each being the
local uniform extension of the union of the preceeding ones. 
We will prove by transfinite induction 
that the maximum modulus principle
for $\cO_{\sfA,\alphaup}$ holds on all $V^{{\text{open}}}\,{\Subset}\,\sfA$ 
for $\alphaup\,{\preceq}\,\,\muup.$ 
We argue by contradiction. If \eqref{eq-3-1} it is not valid,  
then there is a smallest ordinal
$\alphaup\,{\preceq}\,\muup$ for which we can find a domain 
$D\Subset{\sfX}$ such that,
setting $V\,{=}\,D\,{\cap}\,\sfA,$  there is 
an~$u\,{\in}\,\cO_{\sfA,\alphaup}(V){\cap}\cC(\overline{V})$
for which 
\begin{equation}\tag{$*$} \label{eqstar}
1{=}{\max}_{\zq{\in}\bar{V}}|u(\zq)|\;
{>}\;{\max}_{\zq{\in}\partial{D}{\cap}\sfA}|u(\zq)|.
\end{equation} 
Let 
\begin{equation*}
L=\{\zq\in\overline{V}\mid\,|u(\zq)|=1\}\Subset{V}.
\end{equation*}
\par 
%%%%%%%%%%%%%%%%%%%%%%%
By the first part of the proof, we have $\alphaup\,{\succ}\,0,$ 
and the maximum modulus principle
is valid for sections of $\cO_{\sfA,\betaup}$ 
when $0\,{\preceq}\,\betaup\,{\prec}\alphaup.$ 
Fix $\zq_0$ in $L$ with 
$|\zq_{0}|\,{=}\,{\max}_{L}|\zq|.$ 
By a unitary change of coordinates we can assume that \hfill\linebreak
$\zq_{0}{=}(k,0,\hdots,0),$ for a real $k{\geq}0.$  
Fix an open neighbourhood $\omegaup$ of $\zq_{0}$ in
$\sfA,$ with $\omegaup\,{\Subset}\,\sfA,$ 
for which there is a sequence $\{u_{j}\}$ with
$u_{j}\,{\in}\,\cO_{\sfA,\alphaup_j}(\omegaup)
\,{\cap}\cC^0(\bar{\omegaup}),$ with $0{\preceq}\alphaup_j{\prec}
\alphaup$  
and $\sup_{\omegaup}|u(\zq){-}u_{j}(z)|{\leq}2^{-j}.$ 
\par 
The modulus of the 
function $u(\zq)\,{\cdot}\,\exp(\zq^{1}{-}k)$  on $\bar{\omegaup}$ 
has an isolated maximum at $\zq_{0}.$ In particular, if we fix 
$\epsilonup{>}0$
in such a way that 
$\{|\zq{-}\zq_{0}|{\leq}\epsilonup\}{\cap}\sfA{\Subset}\omegaup,$ 
then\par
\centerline{$|u(\zq)\,{\cdot}\,\exp(\zq^{1}{-}k)|{<}1{-}\deltaup$ 
for some $0{<}\deltaup{<}1$ if $\zq\,{\in}\,\sfA$ and $|\zq{-}\zq_{0}|{=}\epsilonup.$ } \par
If we choose $j$ with $2^{1{-}j}{<}\deltaup{\cdot}e^{\epsilonup},$ 
we obtain that, 
for $f_{j}(\zq)\,{=}\,u_{j}(\zq){\cdot}\exp(\zq^{1}{-}k)$, we have
$|f_{j}(\zq_{0})|{>}{\max}_{\zq\in\sfA,\;|\zq-\zq_{0}|
=\epsilonup}|f_{j}|.$
This yields a contradiction, completing the proof of the statement.  
\end{proof}
The final argument in the proof of Prop.\ref{p3.1}
yields 
the following slightly more general  
\begin{cor}\label{supset}
Let $\sfA$ be a locally
closed, set-theoretically weakly one-pseudoconcave subset 
of a reduced complex space $\sfX.$ If $u\,{\in}\,\so_{\sfA}(\sfA)$ and
$|u(\pct)|$ attains a maximum $s$ in $\sfA,$ then  
the set $L\,{=}\,\{\pct\,{\in}\sfA\,{\mid}\,|u(\pct)|\,{=}\,s\}$ is set-the\-o\-ret\-i\-cally weakly pseudoconcave.
\end{cor} 
\begin{proof} 
Assume by contradiction that, under the assumptions in the statement, 
$L$ is not set-theoretically weakly 
pseudoconcave. Then 
there is a point $z_0\,{\in}\,{L}$,
an embedding $\varphi:U\hookrightarrow\C^N$ of some neighbourhood $U$ of $z_0$, with $\varphi(z_0)=0,$
and a smooth strictly plurisubharmonic function $\rhoup$ defined near $0$
such that $\rhoup(0)=0,$ $d\varphi(0)\neq{0}$ and 
$\varphi(L\cap U)\,{\subseteq}\,\{\rhoup\,{\leq}\,0\}.$
After a quadratic coordinate change, we may assume that $\{\varphi\,{<}\,0\}$
is strictly convex at $0$, and even
that it is an open subset of a round sphere in $\C^N.$
Now we got into the situation 
of the final part of the proof of Proposition 
\ref{p3.1},
and can obtain a contradiction 
to the local maximum principle as before.
\end{proof}
Inspection of the proof leads to another way of generalizing Proposition~\ref{p3.1}.
Replacing in the construction of $LACR$ functions $\cO_\sfA$ by some sheaf
$\cF_\sfA=\cF_{A,0}$ of germs of continuous functions on $\sfA$,
we define $\tilde{\cF}_\sfA=\cF_\sfA^{(c)}$ and obtain as before
a chain of intermediate sheaves $\cF_{A,\alpha}.$
The question is to what extent 
the validity of a local maximum principle
may be a  \textit{hereditary} property.
By repeating the arguments
at the end of the proof of Prop.\ref{p3.1} we obtain the following:
\begin{cor}\label{cor3.3}
If a local maximum 
modulus principle is valid for
 an $\cO_{\!\sfA}$-submodule $\cF_{\!\sfA}$ of $\cC_{\sfA}$,
then it is also valid for $\tilde{\cF}_{\!\sfA}$: this means that the following are equivalent: 
\begin{itemize}
 \item[($i$)] for every $\pct\,{\in}\,\sfA$ there is an open
neighbourhood $U_\pct$ of $\pct$ in $\sfA$ such that 
\begin{equation}\label{eq-3-1d}
\forall V^\text{open}\Subset{U}_\pct ,\;\;
\forall u\in\cF_{\sfA}(V)\cap\cC^0(\bar{V}),\;\;
{\sup}_V|u|\leq{\sup}_{\partial_{\sfA}{V}}|u|.
\end{equation}
 \item[($ii$)] for every $\pct\,{\in}\,\sfA$ there is an open
neighbourhood $U_\pct$ of $\pct$ in $\sfA$ such that 
\begin{equation}\label{eq-3-1e}\vspace{-16pt}
\forall V^\text{open}\Subset{U}_\pct ,\;\;
\forall u\in\tilde{\cF}_{\sfA}(V)\cap\cC^0(\bar{V}),\;\;
{\sup}_V|u|\leq{\sup}_{\partial_{\sfA}{V}}|u|.
\end{equation} \qed
\end{itemize}
\end{cor} 
\begin{exam}
If $\sfA$ is a set-theoretically weakly $\qq$-pseudoconcave subset of $\C^n$,
then a local maximum principle is valid for the sheaf $\cF_\sfA^{(\qq)}$ 
obtained by restricting continuous functions which are holomorphic
in the variables $z_\qq,\ldots,z_n.$ 
This follows from Proposition \ref{p3.1} 
since nonempty sections $\sfA\cap\{z_1=c_1,\ldots,z_{\qq-1}=c_{\qq-1}\}$ 
are set-theoretically weakly pseudoconcave, see Lemma \ref{l-1.8}.

If $\sfA$ is a set-theoretically weakly $2$-pseudoconcave subset of $\C^n$,
a local maximum modulus principle is also valid for the sheaf $\cF_\sfA^{[\cdot]}$ obtained by restricting
holomorphic functions and allowing in addition discrete sets of gaps.
More precisely, $u\in\cF_{\sfA,(\pct)}^{[\cdot]}$
if there is a representative $\tilde{u}$ continuous on some relative neighbourhood $U\subset\sfA$ 
of $\pct$ such that $u|_{U\backslash\{\pct\}}\in\cO_\sfA(U\backslash\{\pct\}).$
To show the validity of the local maximum modulus principle, first
observe that at most countably many sections $\sfA\cap\{z_1=c\}$
contain {\textquotedblleft{gap singularities}\textquotedblright} of a given function,
then deduce the maximum modulus principle
along these sections by continuity.
\end{exam} 
\begin{exam} Denote by $\cc(\sfE)$ the set of connected components of a topological
space $\sfE$. \par 
 Let $\sfA$ be a locally closed subset of a reduced analytic complex variety $\sfX.$ 
 We call \emph{c-negligible} a closed subset $\sfB$ of $\sfA$ such that, for  every 
 relatively open subset
 $U$ of $\sfA$, the
 natural map $\cc(U{\setminus}\sfB){\to}
 \cc(U)$ is bijective. Then we can define the \textit{gap sheaf} $\cO_{\!\sfA,[\sfB]}$ by setting 
\begin{equation*}
 \cO_{\!\sfA,[\sfB]}(U)=\cC_\sfA(U)\cap\cO_{\!\sfA}(U{\setminus}\sfB),\;\;\forall U\;\text{open in}\,
 \sfA.
\end{equation*}
Then \eqref{eq-3-1d} is valid for $\cF_{\sfA}\,{=}\, \cO_{\!\sfA,[\sfB]}$.
\end{exam}
\par\medskip
Now we are ready to prove a global maximum 
modulus principle for LACR functions on  locally closed subsets of 
reduced complex spaces.
\begin{thm}\label{t3.6} Let $\sfA$ be a 
set-theoretically weakly $1$-pseu\-do\-con\-cave
 locally 
 closed  
subset of a reduced complex space $\sfX$ 
on which  a global
stictly plurisubharmonic function
$\rhoup$ is defined. 
Then  
for every  
relatively compact 
open subset 
$D$ of $\sfA$  
we have 
\begin{equation}\label{e3.6}
{\max}_{\zq\in 
\overline{D}}\,|f(\zq)|={\max}_{\zq\in
\partial_{\sfA} D}|f(\zq)|,\;\; 
\forall f\in\so_{\!\sfA}({D})\cap\cC^0(\bar{D}).
\end{equation}
\end{thm}
Note that $\rho$ is not required to be an exhaustion function.
Looking at the example where
$\sfX\,{=}\,\C\times\CP^1$ with coordinates 
$(z,(w_0\,{:}\,w_1)),$ 
the weakly $1$-pseudoconcave hypersurface $\sfA\,{=}\R\,{\times}\,\CP^1,$ 
$D\,{=}(-1,1)\,{\times}\,\CP^1\,{\Subset}\,\sfX,$ and 
taking 
$u\,{=}e^{-z^2}$, 
we see that some global assumption on $\sfX$ cannot be avoided.
\begin{proof} 
It is convenient to 
begin by preparing the function $\rhoup.$
By the definition of plurisubharmonic function on complex spaces, 
there is a countable locally finite covering of $\sfX$ by subdomains 
$\sfX_{\! j},$ 
each
equipped with a proper embedding 
$\varphi_j:\sfX_j\hookrightarrow U_j$
into an open subset $U_j$ of $\C^{n_j},$ such that  
each
$\rhoup_j\,{=}\,\rhoup\,{\circ}\,(\varphi_j)^{-1}$ has 
a striclty plurisubharmonic extension 
$\tilde{\rhoup}_j\,{\in}\,\cC^2(U_j).$
By generic and consistent deformations of the $\tilde{\rhoup}_j$'s,
followed by slight shrinking of  $\sfX_j$ and $U_j$, 
we may achieve that the $\tilde{\rhoup}_j$'s are Morse functions.
Replacing $n_j$ by $n_j{+}1$, 
we may even assume that 
the $\tilde{\rhoup}_j$'s do not have critical points, 
after adding an appropriate strictly subharmonic function depending 
only on the new variable.
\par
To argue by contradiction we assume that 
there is a relatively compact open subset 
$D$ of $\sfA$ and a function 
$u\in\so_{\sfA}(D)\cap\cC^0(\bar{D})$ with
$$1={\max}_{\overline{D}}\,|u|>{\max}_{\partial_{\sfA} D}\,|u|.$$
As in the proof of the previous lemma we consider
$L\,{=}\,\{\pct\,{\in}\,\overline{D}\,{\mid}\,|u(\pct)|\,{=}\,1\}\,{\Subset}\, D$
and pick a point $\pct_0\,{\in}\, L$ at which $\rhoup|_L$ 
attains its maximum.
\par
Let $\sfX_{j_0}$ 
be 
a member of the covering $\{\sfX_j\}$ containing $\pct_0.$
The rest of the proof is a local argument dealing only with 
$\sfA'\,{=}\varphi_{j_0}(\sfA\,{\cap}\, \sfX_{j_0})$, 
$L'\,{=}\varphi_{j_0}(L\,{\cap}\,\sfX_{j_0})$
and $\tilde{\rhoup}\,{=}\,\tilde{\rhoup}_{j_0}.$ 
Note that the push forward $\tilde{u}=u\circ(\varphi_j)^{-1}$
is an $LACR$ function.
Certainly we may deform $\tilde{\rhoup}$ so that 
$\{\tilde{\rhoup}\,{=}\,\tilde{\rhoup}(\pct'_0)\}\,{\cap}\, L'\,{=}\,\{\pct'_0\}$.
Hence $L'$ touches the smooth strictly pseudoconvex hypersurface
$\{\tilde{\rhoup}=\tilde{\rhoup}(\pct'_0)\}$ 
from the pseudoconcave side at an isolated point,
and we get a contradiction to Cor.\ref{supset}.
This completes the proof.
\end{proof}
A natural example of a weakly $1$-pseudoconcave set 
can be obtained by taking,
for a compact $K$ in $\C^n,$  
the essential part
$\widehat{K}\backslash K$
 of its polynomial hull, 
 considered as a closed subsets of $\C^n\,{\backslash} K.$
 In this case our maximum modulus principle
geleralizes the local maximum modulus theorem in
Rossi's paper~\cite{rossi60}.  \par\smallskip
Condition \eqref{e3.6} in Theorem~\ref{t3.6} is in fact,
by Proposition~\ref{p3.1}, necessary
and sufficient: 
\begin{thm}\label{thm-3-4}
Let $\sfA$ be  a  locally closed subset 
of a reduced complex space $\sfX$ on which
a strictly plurisubharmonic function $\rhoup$ is defined. 
Then the following properties are equivalent:
\begin{itemize}
\item[\texttt{a})] $\sfA$ is set-theoretically weakly $1$-pseu\-do\-con\-cave.
\item[\texttt{b})] For any domain $D\Subset \sfA$, we have
\begin{equation}\label{maxD'1}
{\max}_{\pct\in\overline{D}}|u(\pct)|
={\max}_{\pct\in\partial_{\sfA} D}|u(\pct)|,
\quad\forall u\in \so_{\sfA}(D)\cap\cC(\overline{D}). 
\end{equation}
\item[\texttt{c})] There is an ordinal $\alphaup\,{\succeqq}\,0$ such that
for any domain $D\Subset \sfA$, we have
\begin{equation}\label{maxD'1}\vspace{-19pt}
{\max}_{\pct\in\overline{D}}|u(\pct)|
={\max}_{\pct\in\partial_{\sfA} D}|u(\pct)|,
\quad\forall u\in \cO_{\sfA,\alphaup}(D)\cap\cC(\overline{D}). 
\end{equation}\qed
\end{itemize}
\end{thm}
\begin{ntz}
We will denote by
$\Psi_{\!\pq}(\sfA)$  the set of points 
where $\sfA$ is
set-theoretically 
$\pq$-pseudoconvex. \end{ntz}
%%%%%%% 
\begin{thm}\label{t-3.8}
Let $\sfA$ be  a  compact subset 
of a reduced complex space $\sfX$ on which
a strictly plurisubharmonic 
function $\rhoup$ is defined.  Then $\Psi_{\! 1}(\sfA)\,{\neq}\,\emptyset$ and
we have 
\begin{equation}\label{e3.9}
 {\max}_{\sfA}|u|={\sup}_{\pct\in\Psi_{1}(\sfA)}|u(\pct)|,\;\;\forall u\in\so_{\sfA}(\sfA\,{\backslash}
 \overline{\Psi_{1}(\sfA)})\cap\cC^{0}(\sfA).
\end{equation}
\end{thm} 
\begin{proof} The point $\op$ where $\rhoup$ attains its maximum on $\sfA$ belongs to
$\Psi_{1}(\sfA),$ which therefore is nonempty. If $\sfA\,{=}\,\overline{\Psi_{1}(A)}$, then
there is nothing to prove. Assume that $\sfA\,{\neq}\,\overline{\Psi_{1}(A)}$
and let $u\,{\in}\,\so_{\sfA}(\sfA\,{\backslash}
 \overline{\Psi_{1}(\sfA)})\cap\cC^{0}(\sfA).$ Set
\begin{equation*}
 L=\{\pct\,{\in}\,\sfA\mid |u(\pct)|={\max}_{\sfA}|u|\}.
\end{equation*}
By Cor.\ref{supset}, if $L\,{\cap}\,\overline{\Psi_{1}(\sfA)}\,{=}\,\emptyset,$ then
$L$ is set-theoretically 
$1$-pseudoconcave.   This yields a contradiction, because
the compact subset
$L$ is set-theoretically $1${-}pseu\-do\-convex at points where the restriction of $\rhoup$
to $L$ has a maximum. Hence $L\,{\cap}\,\overline{\Psi_{1}(\sfA)}{\neq}\emptyset$ 
and therefore \eqref{e3.9} holds true.
\end{proof}
We get a slight strengthening of \eqref{e3.9} by setting
 $\sfB{=}\sfA\backslash\mbox{int}_{\sfA}(\Psi_{1}(\sfA)),$
where $\mbox{int}_{\sfA}$ denotes the relative interior.
Since
$\partial_{\sfA}\Psi_{1}(\sfA){=}\Psi_{1}(\sfA)\backslash\mbox{int}_{\sfA}(\Psi_{1}(\sfA))$ 
is contained in $\Psi_{1}(\sfB)$, we obtain 
\begin{equation*}
 {\max}_{\sfB}|u|={\max}_{\partial_\sfA\Psi_{1}(\sfA)}|u|,\;\;\forall u\in\so_{\sfA}(\sfA\,{\backslash}
 \overline{\Psi_{1}(\sfA)})\cap\cC^{0}(\sfB).
 \end{equation*}

\section{Maximum modulus principle $\mathrm{I{I}}$}
\label{sec-4}
We may use Theorem~\ref{t3.6} as 
a first step in the proof of 
a maximum modulus principle for LACR functions on locally
 closed subsets of  reduced
complex spaces satisfying weaker global assumptions.

\begin{thm}\label{t4.1} Let $\pq,\qq$ be positive integers, 
 $\sfX$ 
 an $n$-dimensional complex manifold, 
 on which one can define a global
$\cC^2$-smooth real valued function whose
complex
 Hessian 
has at each point at least $n{-}\pq{+}1$ positive
eigenvalues, and  
$\sfA$  a set-theoretically weakly $\qq$-pseu\-do\-con\-cave
locally
 closed subset of $\sfX$. If $\qq{\geq}\pq,$ 
 then 
the maximum modulus
principle for $LACR$ functions 
holds on $\sfA$. This means that, 
for every  relatively compact domain $D{\Subset}\sfA,$
\begin{equation}\label{e4.1} 
{\max}_{\pct\in\overline{D}}|u(\pct)|
={\max}_{\pct\in\partial_{\sfA} D}|u(\pct)|,
\quad \forall{u}\in\so_{\sfA}(D)\cap\cC(\Bar{D}).
\end{equation}
\end{thm} 
\begin{proof}
 We argue by contradiction. 
 Assume that there is a relatively open subset $D$ of $\sfA$ with
 $D\,{\Subset}\,\sfA$ for which there is a function 
 $u\,{\in}\,\so_{\sfA}(D)\cap\cC(\Bar{D})$
for which \eqref{e4.1} is not valid. Then \vspace{-3pt}
\begin{equation*}
 L=\{\pct\in{D}\mid |u(\pct)|={\max}_{\overline{D}}|u|\}
\end{equation*}
is compact and therefore there is $\pct_{0}$ in $L$ with 
$\rhoup(\pct_{0})\,{=}\,{\max}_{L}\rhoup.$ 
We choose local coordinates $(U,z)$ centered at 
$\pct_{0}$ such that 
on the $(n{-}\pq{+}1)$-dimensional complex linear 
subspace~$\sfY\,{=}\,\{z^i{=}0,\; 1{\leq}i{\leq}\pq{-}1\}$ 
the complex hessian of $\rhoup$ 
is positive definite. \par
By Lemma~\ref{l-1.8} we know that 
$\sfB{\coloneqq}\sfA\,{\cap}\,\sfY$ is set-theoretically weakly pseudoconcave. 
But the restriction of
$\rhoup$ to $\sfY$ is strictly plurisubharmonic and then 
$L{\cap}\sfB\,{\subset}\,\sfY\,{\cap}\,
\{\rhoup{\leq}\rhoup(\pct_{0})\}$
shows that $L\,{\cap}\,\sfB\,{=}\,\{\pct\,{\in}\sfB\,{\mid}\,
|u(\pct)|\,{=}\,|u(\pct_{0})|\}$ is not set-theoretically weakly pseudoconcave, contradicting
Cor.\ref{supset}.
\end{proof}
%%%%%%%%
If $K$ is a compact subset of $\C^n,$
then  
every holomorphic curve  
with boundary in $K$ is contained in its
polynomial hull.
Theorem \ref{thm-3-4} yields the following generalization 
of this fact:
\begin{cor}
Let $M$ be  a locally closed
$\cC^2$-smooth weakly $1$-pseu\-do\-con\-cave submanifold
of $\C^n$ and $K{\Subset}\C^n$ compact. Then every component
of $M\backslash K$ which is relatively compact in $M$ is contained
in $\widehat{K}$.\qed
\end{cor}
%%%%%%
\begin{thm}\label{t-4.3}
Let $\pq$ be a positive integer and 
$\sfX$ an $n$-dimensional complex manifold 
 on which one can define a global
$\cC^2$-smooth real valued function whose 
complex Hessian 
has at each point at least $n{-}\pq{+}1$ positive
eigenvalues and  
$\sfA$  a compact subset of $\sfX.$  
Then $\Psi_{\!\pq}(\sfA)\,\neq\,\emptyset$ and 
\begin{equation}\label{eq-4.2}
 |u(\pct)|\leq {\sup}_{\Psi_{\!\pq}(\sfA)}|u|,\;\;\;
 \forall u\in\so_{\sfA}(\sfA\,{\backslash}\overline{\Psi_{\!\pq}(\sfA)})
 \cap\cC^{0}(\sfA).
\end{equation}
\end{thm} 
\begin{proof} 
Let $\rhoup$ be a $\cC^{2}$ smooth real valued function on $\sfX$ 
whose
complex hessian has everywhere $n{-}\pq{+}1$ positive eigenvalues. 
 The points of $\sfA$ where $\rhoup$ 
 has a maximum make a nonempty subset
 of $\Psi_{\!\pq}(\sfA).$ If $\sfA\,{=}\,\overline{\Psi_{\!\pq}(\sfA)},$ 
 then there is nothing to prove.
 \par
 Assume that  $\sfA\,{\neq}\,\overline{\Psi_{\!\pq}(\sfA)}$ and let 
$u\,{\in}\,\so_{\sfA}(\sfA\,{\backslash}\overline{\Psi_{\!\pq}(\sfA)})
 \cap\cC^{0}(\sfA).$ Set
\begin{equation*} L=\{\pct\,{\in}\,\sfA\mid |u(\pct)|
={\max}_{\sfA}|u|\}.\end{equation*}
Assume by contradiction that 
$L\,{\cap}\,\overline{\Psi_{\!\pq}(\sfA)}\,{=}\,\emptyset.$  
Fix a point $\op$ in $L$ 
where the restriction of $\rhoup$ to $L$ has a maximum 
and take
an $(n{-}\pq{+}1)$-dimensional complex submanifold 
$\sfY$ of $\sfX$ containing $\op$ 
and such that the restriction of $\rhoup$ to $\sfY$ 
has a positive definite complex hessian at $\op.$ 
The intersection $\sfY\,{\cap}\, 
(\sfA\,{\backslash}\overline{\Psi_{\!\pq}(\sfA)})$ is a  closed
subset of $\sfY$ which is set-theoretically pseudoconcave by Lemma~\ref{l-1.8}.
Then $L\,{\cap}\,\sfY$ should be set-theoretically pseudoconcave at all points because
of Cor.\ref{supset}. This contradicts the fact that $L\,{\cap}\,\sfY\,{\subseteq}\{\rhoup{\leq}0\}.$
We proved that $L\,{\cap}\,\overline{\Psi_{\!\pq}(\sfA)}\,{\neq}\,\emptyset$
and this implies~\eqref{eq-4.2}.
\end{proof}
An analogous refinement as the one following Thm.\ref{t-3.8} is valid.

\section{Applications to $CR$ manifolds.}\label{sec4} 
Recall that a $\cC^2$-smooth real submanifold 
$M$ of a complex manifold $\sfX$ is
$CR$ 
if the dimension of $H_\pct M=\mathrm{T}_\pct M\cap{J_\pct {\mathrm{T}_\pct M}}$ 
 (here $J{:}\mathrm{T}\sfX{\to}{T}\sfX$ is the almost complex structure of $\sfX$)
does
not depend on \mbox{$\pct\in M.$}
If $M$ is a $CR$ submanifold, then
$HM={\bigcup}_{\pct\in{M}}H_{\pct}M$ is a real
vector subbundle of $TM.$  
The \emph{characteristic bundle} $H^0 M\subset \mathrm{T}^* M$ 
of 
$M$ is 
the annihilator bundle of $HM.$ 
The restriction to $HM$ of the complex structure $J$ of $\sfX$ 
defines a complex structure on the fibers of $HM$ 
and we call the integers
\begin{equation*}
\text{$n_M={\rk}_{\C} H M$ 
\;\;and \;\;
 $d_M=\rk_{\R}H^0 M
\! =\!
\dim_{\R} M-2{n}_M$}\end{equation*}
its 
$CR$-dimension and $CR$-codimension, respectively.
The \emph{vector-valued Levi form} 
$\cL_M:HM\to{TM}{/}HM$
is 
defined by
\begin{equation*}
H_{\pct}M\ni X \to
\cL_{M} (X)=[J\widetilde{X},\widetilde{X}]({\pct}), \!\!\!\mod 
H_{\pct} M,
\end{equation*}
where $\widetilde{X}$ is any smooth 
$HM$-valued section extending 
$X$.
For $\xiup\in H^0_{\pct}M$, the \emph{$\xiup$-directional Levi form} is 
\[
\cL_{M}^\xiup (X)=\langle \xiup,[J\widetilde{X},
\widetilde{X}]({\pct})\rangle
=\langle{\xiup},\cL_M(X)\rangle.
\]
\begin{defn}\label{defn1.3}
We say that $M$ is
\emph{weakly $\qq$-pseudoconcave at $\pct\in M$} if for every
$\xiup\in H^0_{\pct} M$ the directional Levi form $\cL_{M}^\xiup$
has at least $\qq$ nonpositive eigenvalues.
\end{defn} 
The 
notions of weak pseudoconcavity 
of Definitions \ref{defn1.1a}
and \ref{defn1.3} 
coincide for $\cC^2$-smooth $CR$ submanifolds:
%%%%%%%
\begin{lem}\label{lem-1-22}
An embedded CR manifold $M$ of class $\cC^2$ 
is weakly $q$-pseu\-do\-con\-cave at $\pct\in M$
if and only if it is set-theoretically weakly $q$-pseu\-do\-con\-cave at $\pct$.
\end{lem}
%%%%%%%%%%
\begin{proof} 
Let $\sfX$ be an $n$-dimensional complex manifold and
$M$ a $CR$ submanifold of class $\cC^2$ of $\sfX.$
If $M$ is not set-theoretically weakly $q$-pseudoconcave at $\pct$,
then 
there are an open neighbourhood $U$ of $\pct$ in $\sfX$
and an $(n{-}q{+}1)$-pseudoconvex 
real valued smooth function $\phiup$ on $U$ 
such that $d\phiup(\pct){\neq}0$ and 
$M{\cap}U{\subset}\{\phiup{<}0\}{\cup}\{\pct\}.$ 
Then $\xiup={d}^c\phiup(\pct)\,{\in}\,H^0_{\pct}(M)$ 
and $\cL_M^{\xiup}$
is the restriction of $\iq^{-1}\partial\bar{\partial}\phiup(\pct)$ to $H_{\pct}M.$ 
This Hermitian symmetric form has at least $(n_M{-}\qq{+}1)$ positive eigenvalues on
$H_{\pct}M$ and therefore $M$ is not weakly $\qq${-}pseudoconcave at 
$\pct.$ 
\par 

Conversely, if $M$ is not weakly $q$-pseudoconcave at $\pct$, then
there is $\xiup\in H^0_{\pct} M$ and a complex
subspace $W$ of dimension $(n_M{-}q{+}1)$ 
of $H^0_{\pct} M$ on which $\cL_M^\xiup$ is 
positive definite.
Let $\xiup=d^c\rhoup(\pct)$ for a smooth real valued function
vanishing on $M\cap{U}_{\pct}.$ We can assume that 
$U_{\pct}$  is the domain of a 
coordinate chart
$(U_{\pct},z)$ in $\sfX,$ centered at $\pct,$   
where $M$ is described as the set of common zeros of real-valued 
$\cC^2$-smooth functions
$\rhoup_1,\hdots,\rhoup_{d_M},$ with $\rhoup_1\,{=}\,\rhoup$ 
and $\partial\rhoup_1(\pct),\hdots,\partial\rhoup_{d_M}(\pct)$
linearly independent. 
Then, having fixed holomorphic coordinates 
$\zq_1,\hdots,\zq_n$ centered ad $\pct,$ the function 
\begin{equation*}
 \phiup(\zq)=\rhoup_1{+}{\sum}_{i=1}^{d_M}\rhoup_i^2
 -{\epsilonup}|\zq|^2
\end{equation*}
has at $\pct,$ for small $\epsilonup>0,$  
a complex Hessian with at least
$n{-}q{+}1$ positive eigenvalues
and we can find an open neighbourhood $U$ of $\pct$ such that 
$M\cap{U}\subset\{\phiup{<}0\}{\cup}\{\pct\}.$ 
It suffices to approximate $\phiup$ 
with a smooth real valued function $\psiup$ with
$|\phiup-\psiup|{\leq}\deltaup|\zq|^2$ with $0{<}\deltaup{\ll}1$ 
and $\psiup(0){=}0,$ $d\psiup(0){=}d\phiup(0),$ 
 to show that $M$ is 
not 
 set-theoretically weakly $\qq$-pseudoconvex
at $\pct.$  
The lemma is proved.
\end{proof}
Recall that  a complex valued $\cC^1$-smooth function  is 
$CR$  on $M$ if its differential vanishes on 
$$\mathrm{T}^{0,1} M=(\C\otimes_M HM)\cap \mathrm{T}^{0,1} \sfX.$$
If $M$ is sufficiently smooth, this condition can be interpreted
in the weak sense in order to introduce the sheaves
$\cO^0_M$ of germs of continuous $CR$ functions and 
$\cO^{-\infty}_M$ of germs of $CR$ distributions on $M.$
If $M$ is $\cC^2$-smooth, 
the Baouendi-Treves approximation theorem (\cite{ba-tr}) 
implies that 
$\cO^0_M=\sO$.
If we only assume that
$M$ is of class $\cC^1$, then
we still have 
the inclusion
$\cO^1_M\subseteq\sO,$
but there are both cases in which 
the two sheaves are equal and cases
whether they are not.
Moreover, the 
size 
of the domains 
on which a given $LACR$ function can be uniformly approximated
by restrictions of holomorphic functions defined in the ambient space
may depend on the individual function,
whereas, according to the Baouendi-Treves approximation theorem,
 it may be taken to depend only upon 
 its domain of definition if $M$ is a $CR$ manifold of class $\cC^2.$
\par
Theorems \ref{t3.6}, \ref{thm-3-4} and \ref{t4.1} 
yield as a corollary
maximum modulus principles for continuous CR functions 
on embedded CR manifolds of class $\cC^2$.
For $M$ of class $\cC^1$, one still obtains 
maximum modulus principles for CR functions which
are $\cC^1$-smooth and continuous up to
the boundary.
For later reference, 
we state the following consequence 
for CR submanifolds of Stein manifolds.\par 
We denote by $M_{\mathrm{pcx}}$ the set of points of
a $CR$ submanifold $M$ of 
an $n$-dimensional complex manifold 
$\sfX$ at which $M$ is $1$-pseudoconvex: \begin{equation*}
M_{\mathrm{pcx}}=\{\pct\in{M}\mid \exists\,\xiup\in{H}^0_{\pct}M,\;\;
\text{with}\;\; \cL_M^{\xiup}>0\}.
\end{equation*}

\begin{cor}\label{X Stein}
Let $M$ be a CR submanifold of a complex manifold $\sfX$
on which a stricly plurisubharmonic
 function $\rhoup$ is defined.
\begin{enumerate}
\item[{\bf a)}]
For every domain $D\Subset M$, we have
\begin{equation}\label{maxDa}
{\max}_{z\in\overline{D}}|u(z)|={\max}_{z\in\partial_{M} D\cup 
(D\cap\overline{M}_{\mathrm{pcx}})}
|u(z)|,\quad\forall u\in{\cO}_{M}^{0}(D)\cap\cC(\overline{D}).
\end{equation}
\item[{\bf b)}] The maximum modulus principle 
\eqref{maxD'1}
 holds for every domain $D\Subset M$
if and only if $M$ is weakly $1$-pseu\-do\-con\-cave.\qed
\end{enumerate}
\end{cor}

This result was first proved  by A.~Daghighi and the second author
(see the PhD thesis \cite{dagh}),
as the optimal one 
in a longer chain of results for embedded CR manifolds.
After earlier results in \cite{Ior86} 
(mainly sufficient conditions for the maximum 
modulus
principle
in terms of barrier varieties),
Corollary \ref{X Stein} was proved in \cite{Si}
for $M$ of class $\cC^4$ and $u$ of class $\cC^3$.
More recently, a refinement for functions of class $\cC^2$ 
was obtained in \cite{BW}.

\section{Holomorphic extension to  
ambient neighbourhoods}
\label{sect6}
In Example \ref{e-1.9} we constructed smooth set-theoretically weakly $(n{-}1)$-pseudo\-concave 
submanifolds $\sfM_{k}$ of 
codimension $2$ in 
$\C^{2n+1}\! ,$ having an isolated CR singularity at the origin.
Those
$\sfM_{k}$ are actually strictly 
$(n{-}1)$-pseu\-do\-con\-cave at all points 
$\pct\,{\neq}\,0$ and hence,  if $n\,{\geq}\, 2,$
all $CR$ functions
defined on relatively open
subsets of $\sfM_{k}\,{\backslash}\{0\}$  holomorphically
extend to open neighbourhoods 
in the ambient space $\C^{2n+1}\!$. A natural 
question is whether this is still valid for $LACR$ 
functions defined on a neighbourhood of $0$ in $\sfM_{k}.$
Thm.\ref{t6.1} will answer this question in the affirmative.
\par\smallskip
Let $\sfM$ be a smooth 
real submanifold of a complex manifold $\sfX.$
Following \cite{AHNP2010}, we say that 
\begin{defn}
$\sfM$ has the \emph{holomorphic extension property at 
its point $\pct$}
if every $LACR$ function defined on $\sfM$ near $\pct$ has a holomorphic extension
to an open neighbourhood of $\pct$ in $\sfX.$ \end{defn}
Sufficient conditions 
for the validity of the holomorphic
extension property in the case of $CR$ manifolds 
can be found in
\cite{AHNP08a,NP2015}.
Here we extend some of these results to smooth real submanifolds
carrying $CR$ singularities.
\subsection{Isolated $CR$ singularities} \label{s7.1}
We begin by proving a result on holomorphic extension at isolated $CR$ singularities.
%considering the case
%of smooth submanifolds carrying isolated $CR$ singularities.
\begin{thm}\label{t6.1}
 Let $\sfM$ be a %set-theoretically 
 weakly $1$-pseudoconvave locally closed smooth real submanifold 
 of a complex manifold $\sfX$ and $\sfE$ 
 a discrete subset of $\sfM$. Assume that \begin{itemize}
 \item 
 the set $\sfM_{\text{reg}}$ of $CR$-regular points of $\sfM$
 is open and dense in $\sfM$ and is a generic $CR$ submanifold of $\sfX$;
 \item 
$\sfM\,{\backslash}\sfE$  has at each point the holomorphic extension property. 
\end{itemize}\par 
Then, for each relative domain $\omegaup\,{\subset}\,\sfM,$
there is an ambient domain $\Omega\,{\subset}\,\sfX$ with $\omegaup\,{=}\,\Omega\cap M$
such that all $LACR$ functions on $\omegaup$ 
have a unique holomorphic extension to $\Omega.$
In particular, every $LACR$ function defined on some open subset of $\sfM$ is $\cC^\infty$-smooth
and we have~${\tiO}_{\sfM}\,{=}\,\cO_{\sfM,0}.$
\end{thm}
%\begin{rmk}
% By \cite{NP2015} the second assumption of Thm.\ref{t6.1} is equivalent to the fact
% that every continuous $CR$ function on an open subset of $\sfM_{\text{reg}}$
% is $\Ci$-smooth.
%\end{rmk}
\begin{proof} By the assumption that the embedding $\sfM\,{\hookrightarrow}\sfX$ is almost everywhere
generic, a holomorphic function defined on a connected open subset $\Omega$
of $\sfX$ intersecting $\sfM$
is uniquely determined by its values on $\sfM\,{\cap}\,\Omega.$ 
Thus the statement of the Theorem is equivalent to the fact that
${\tiO}_{\sfM}\,{=}\,\cO_{\sfM,0}$ and this, 
by the recursive definition of $LACR$ functions, 
 is equivalent to $\cO_{\sfM,1}\,{=}\,\cO_{\sfM,0}.$ 
\par 
The holomorphic extension property implies this at  points of $\sfM{\setminus}\sfE$.
Thus it suffices to show 
that 
$(\tiO_{\sfM})_{(\pct)}=(\cO_{\sfM,0})_{(\pct)}$
at points 
$\pct$ %%%%%
of 
$\sfE$,
permitting us
to localize, 
reducing  
to the case in which 
$\sfX$ is the unit ball in $\C^n$
and~$\sfE\,{=}\,\{0\}\,{\subset}\,\sfM$. 
\par
Stalkwise coincidence of $\cO_{\sfM,1}$ and $\cO_{\sfM,0}$ will follow
as soon as we have shown that $\cO_{\sfM,1}(\omegaup)\,{=}\,\cO_{\sfM,0}(\omegaup)$
for a family of  
arbitrarily small open neighbourhoods $\omegaup$ of $0$ in $\sfM.$ 
For $u\,{\in}\,\cO_{\sfM,1}(\omegaup)$ we can find a smaller connected neighbourhood 
$\omegaup_1{\Subset}\,\omega$ of $0$ in $\sfM$ such that  
$u$ can be uniformly 
approximated on $\bar{\omegaup}_{1}$ by a sequence 
$\{\vq_{j}\}\,{\subset}\,\cO_{\sfM,0}(\bar{\omegaup}_1).$
\par 
  A germ of holomorphic function $f$ at $\zetaup\,{\in}\,\sfX$
 is the sum of its Taylor series at $\zetaup$
\begin{equation*}
 f(z)={\sum}_{\alphaup\in\N^{n}}\frac{D_{z}^{\alphaup}f(\zetaup)}{\alphaup!} (z-\zetaup)^{\alphaup},
\end{equation*}
uniformly converging 
on the  balls $\bar{B}_{\zetaup}(r)\,{=}\,\{|z{-}\zetaup|{\leq}r\}$
which are contained in its domain of definition 
and its holomorphic derivatives at $\zetaup$ satisfy  
\begin{equation*}\tag{$*$}
 |D_{z}^{\alphaup}f(\zetaup)|\leq C\,\frac{\alphaup!}{r^{|\alphaup|}},\;\;\forall\alphaup\in\N^{n},
\end{equation*}
with $C\,{=}\,\sup_{|z-\zetaup|=r}|f(z)|.$ 
Vice versa, if complex numbers $c_{\alphaup}$ are given, which satisfy 
\begin{equation*}\tag{$**$}
 |c_{\alphaup}|\leq C\,\frac{\alphaup!}{r^{|\alphaup|}},\;\;\forall\alphaup\in\N^{n},
\end{equation*}
for some constant $C{\geq}0,$ 
then the series
\begin{equation*}
 {\sum}_{\alphaup\in\N^{n}}\dfrac{c_{\alphaup}}{\alphaup!} (z-\zetaup)^{\alphaup}
\end{equation*}
converges, uniformly on compact subsets, to a function $f$ which is holomorphic on the ball
$B_{\zetaup}(r)=\{|z-\zetaup|{<}r\},$ with $D^{\alphaup}_{z}f(\zetaup)=c_{\alphaup}$ for
all multi-indices~$\alphaup.$  \par
Fix a connected open neighbourhood $\omegaup_{2}$ of $0$ in $\sfM$ 
with $\omegaup_{2}\,{\Subset}\,\omegaup_{1}.$ 
The approximating functions $\vq_{j}$ holomorphically extend to open neighbourhoods
of the compact balls $\bar{B}_{\zetaup}(r)$ for all $\zetaup\,{\in}\,\partial_{\sfM}\omegaup_{2}$
and a same 
sufficiently small $r{>}0.$ 
The fact that this $r$ can be taken independent of $j$ follows indeed from a simple
functional analysis argument, since $\cO_{\sfM,1}(\omegaup\,{\backslash}\{0\})$ 
is a Fr\'echet space
which is equal to the union of the images of the continuous maps 
\begin{align*}
 \mathfrak{F}_{\nuup}
 {=}\{(u,f){\in} \cO_{\sfM,1}(\omegaup_{1}\,{\backslash}\{0\}){\times}\cO(B_{\zetaup}(2^{{-}\nuup})\,{\mid}\,
 f(z)\,{=}\,u(z),\, \forall z\in\omegaup_{1}{\cap}{B}_{\zetaup}(2^{-\nuup})\}\,{\ni}\, (u,f)\;\,
 \\
 \big\downarrow{\piup_{\nuup}}\;\\
  \cO_{\sfM,1}(\omegaup_{1}\,{\backslash}\{0\})\ni u.\;\;\;
\end{align*}
Since each $\mathfrak{F}_{\nuup}$ is a Fr\'echet subspace of the product
$\cO_{1}(\omegaup_{1}{\backslash}\{0\})\times\cO(B_{\zetaup}(2^{\nuup})),$ 
Baire's theorem implies that some $\piup_{\nuup_{\zetaup}}(\mathfrak{F}_{\nuup_{\zetaup}})$ 
is of the second Baire category
and hence equal to $\cO_{1}(\omegaup_{1}\,{\backslash}\{0\})$ by Banach's open mapping theorem.
Finally we find our uniform $r{>}0$ by taking a finite covering of the compact set $\partial_{\sfM}\omegaup_{2}$
by balls $B_{\zetaup}(2^{-\nuup_{\zetaup}})$ and taking $r{>}0$ so small that the $r$-neighbourhood
of $\partial_{\sfM}\omegaup_{2}$ is contained in their union.
\par 
Moreover,
the holomorphic derivatives of the $\vq_{j}$'s satisfy 
estimate $(*)$, with a uniform $C{\geq}0,$  
because by the holomorphic extension properties the values of their holomorphic extensions
to $\bar{B}_{\zetaup}(r),$ for $\zetaup\,{\in}\,\partial\omegaup_{2}$ and $r{\ll}1,$ 
are already taken on $\omegaup_{1}.$ 
\par
This implies that the coefficients of the Taylor series at $0$ of 
all $\vq_{j}$'s   satisfy a uniform estimate $(**).$ 
By the maximum modulus principle 
the restrictions to $\omega_{1}$ of the holomorphic derivatives of the $\vq_{j},$
which are LACR functions on $\omega_{1},$ satisfy $(**)$ at all points of
$\omega_{2}.$ In particular, 
the $\vq_{j}$'s
extend to holomorphic functions $\tilde{\vq}_{j}$'s, 
which are defined and uniformly bounded on compact subsets of $B_{0}(r).$ 
This implies that also their uniform limit is, on $M\,{\cap}\,B_{0}(r),$ the restriction
of a function which is holomorphic on $B(0,r).$ 
\end{proof}
\begin{rmk} In case the set $\sfM^{\sharp}$ of $CR$-singular points of $\sfM$ is
contained in $\sfE$, 
 by \cite{NP2015} the second assumption of Theorem\,\ref{t6.1} is equivalent to the fact
 that every continuous $CR$ function on an open subset of $\sfM\,{\backslash}\sfE$
 is $\Ci$-smooth.
\end{rmk}

\begin{thm}\label{t6.3}
Let $\sfM$ be a smooth real 
submanifold of a complex manifold $\sfX$
and $\sfE$ a discrete subset of $\sfM$ 
containing its $CR$ singularities.
Let $\sfE_{\rm reg}\,{\subset}\,\sfE$ be the subset of 
points of $\sfE$ where $\sfM$ is $CR$ regular.\par
Assume that:
\begin{itemize}
 \item $\sfM\,{\backslash}\sfE$ is a generic 
 weakly 
$1$-pseu\-do\-con\-cave $CR$ submanifold of $\sfX\,{\backslash}\sfE$;
\item $\sfM\,{\backslash}\sfE_{\textrm{reg}}$ has the holomorphic extension property at each of its points.
\end{itemize}
\par

Then for each relative domain $\omegaup\,{\subset}\,\sfM,$
there is an ambient domain $\Omega$ with $\omegaup\,{=}\,\Omega\cap \sfM$
such that all $LACR$ functions on $\omegaup\,{\backslash}\sfE_{\rm reg}$ 
have a unique holomorphic extension to $\Omega.$
\end{thm}

\begin{proof}
Localizing as in the proof of Thm.\ref{t6.1}, we 
reduce to 
the case where 
$\sfM$ is a generic $CR$ submanifold of $\C^{n}$,
$\sfE\,{=}\,\sfE_{\rm reg}\,{=}\,\{0\}$
and the LACR function $u$ is defined on $\sfM\backslash\{0\}$. %%%%
By assumption $CR$ functions on $\sfM\backslash\{0\}$ 
have holomorphic extension to an ambient neighbourhood $\Omega$ of $\sfM\backslash\{0\}.$
Since isolated points are 
wedge removable (see e.g. \cite{MP1}), 
there is a wedge-like domain $\mathcal V$ attached
to all of $\sfM$ and a wedge-like domain ${\mathcal{V}'}\,{\subset}\,{\mathcal{V}}\cap\Omega$
attached to $\sfM\backslash\{0\}$
such that holomorphic functions $u$ on $\Omega$ 
have 
holomorphic extensions $\tilde{u}$ 
to $\mathcal{V}$
agreeing 
on ${\mathcal{V}'}$.
\par
Let now $\vq$ be a unit vector at $0$ pointing into $\mathcal V$ 
and $\omegaup\,{\Subset}\,\Omega$ a small neighbourhood of $0$ in $\sfM.$
For small $\epsilonup\,{>}\,0$, the function $u_{\epsilonup}(z)\,{=}\,\tilde{u}(z\,{+}\,\epsilonup{\cdot}\vq)$
is defined and holomorphic on a neighbourhood of ${\omegaup}$. 
By 
the argument in the proof of %%%
Thm.\ref{t6.1}, there is $\etaup\,{>}\,0$, only depending on $\omegaup$,  
such that $u_{\epsilonup}$ extends to a holomorphic
function on $B_{\etaup}({0})$. This means that $\tilde{u}$ holomorphically extends to 
$B_{\epsilonup{\cdot}\vq}(\etaup)$
for all sufficiently small $\epsilonup\,{>}\,0$ and hence that it extends holomorphically on a neighbourhood
of~$0$.
\end{proof}

For achieving  the proof of Thm.\ref{t6.3} it suffices that the extension property is valid
at points of $\partial\omegaup$. Hence the following variant follows from the same argument.
\begin{cor}\label{c6.4}
Let $\sfM$ be a smooth real submanifold of a complex manifold~$\sfX$,
$\sfE$ a discrete subset of $\sfM$ and $\sfA$ a closed subset of $\sfM{\setminus}\sfE$.
Assume that: 
\begin{itemize}
 \item $\sfM\,{\backslash}\sfE$ is a generic set-theoretically weakly 
$1$-pseu\-do\-con\-cave $CR$ submanifold of $\sfX\,{\backslash}\sfE$;
\item $\sfM\,{\backslash}\sfE$ has at each point 
the holomorphic extension property; 
\item $\sfA$ is $\mathcal{V}$-removable.
\end{itemize}
\par
Then for each relative domain $\omegaup\,{\subset}\,\sfM,$ 
there is an ambient domain $\Omega$ with $\omegaup\,{=}\,\Omega\cap \sfM$
such that all $LACR$ functions on $\omegaup{\backslash} \sfA$ 
have a unique holomorphic extension to $\Omega.$\qed
\end{cor}
\par
We obtain here a generalization of a result of \cite{NP2015}:
\begin{thm}\label{t6.7}
Let $\sfM$ be a smooth real 
submanifold of a complex manifold $\sfX$
and $\sfE$ a discrete subset of $\sfM$ 
containing its $CR$ singularities.
Then the following are equivalent: 
\begin{enumerate}
 \item All germs of $CR$ functions 
 at points of $\sfM\,{\backslash}\sfE$ are $\Ci$-smooth;
 \item All germs in $\tiO_{\sfM}$ are $\Ci$-smooth;
 \item $\sfM$ has the holomorphic extension property.
\end{enumerate}
\end{thm} 
\begin{proof}
 By \cite[Thm.1.1]{NP2015} $(1)$ is equivalent to the fact that $\sfM\,{\backslash}\sfE$ has
 the holomorphic extension property. By Cor.\ref{c6.4} this implies $(3).$ Since it is clear that
 $(3){\Rightarrow}(2){\Rightarrow}(1)$, this completes the proof.
\end{proof}
\par\subsection{Massive sets of $CR$ singularities} 
In \S\ref{s7.1} we showed connections between holomorphic extension property
and \textit{wedge removability}. 
Actually removable sets can be quite \textit{massive}. %, see [????]. 
For example, it was shown in \cite{MP2} that, in the situation of Cor.\ref{c6.4},
each $\sfE$ 
with zero $(\dim_{\R}(\sfM){-}2)$-dimensional Hausdorff volume
is removable.\par 
Here we prove some further result of holomorphic extendability 
by using 
the results of 
\S\ref{s7.1} 
and 
Lemma~\ref{l-1.8}.
\par\smallskip
Let $\sfE$ be a locally closed subset of an
 $n$-dimensional complex manifold $\sfX.$ Fix an  
integer
$d,$ with $0{\leq}d{\leq}n$ 
and a point $\pct$ of $\sfE.$ 
\begin{ntz}
  Let $\mathcal{Y}(\sfE,\pct,n{-}d)$ be the set of germs
$(\sfY,\pct)$ of smooth $(n{-}d)$-dimensional complex submanifolds of $\sfX$ such that
$\pct$ is an isolated point of $\sfE\,{\cap}\,\sfY$.

\end{ntz}
We have a natural map 
\begin{equation}\label{e6.1}
\piup_{\pct}: \mathcal{Y}(\sfE,\pct,n{-}d)\ni (\sfY,\pct)\to \mathrm{T}^{1,0}_{\pct}\sfY
\in \Gr^{\C}_{\! n-d}(\sfX,\pct),
\end{equation}
where we denoted
by $\Gr^{\C}_{\! n-d}(\sfX,\pct)$ 
the Grassmannian of complex $(n{-}d)$-di\-men\-sional linear subspaces of $\mathrm{T}^{1,0}_{\pct}\sfX.$
\begin{defn}\label{d6.2}
We say that a subset $\sfE$ of an $n$-dimensional  
complex manifold $\sfX$ 
has \emph{transversal complex dimension
$d$} at its point $\pct$ if $d$ is the smallest nonnegative integer for which 
\begin{itemize} 
\item 
$\mathcal{Y}(\sfE,\pct,n{-}d)$ is not empty;  
\item 
$\piup_{\pct}(\mathcal{Y}(\sfE,\pct,n{-}d))$ is dense in 
an open subset of $\Gr^{\C}_{\! n-d}(\sfX,\pct).$
\end{itemize}
\par
We say that $\sfE$ has \emph{transversal complex dimension less or equal
to $d$} when its transversal dimension is less or equal $d$ 
at all points. 
\end{defn}
\begin{exam}\label{e6.6}
 An analytic subvariety  of pure dimension 
 $d$ of $\sfX$ has complex transversal dimension $d.$
 A smooth $CR$ submanifold of $\sfX$ having $CR$ dimension $m$ and $CR$ codimension
 $k$ has transversal complex dimension $m{+}\!\left[\frac{k{+}1}{2}\right].$ 
\end{exam}
\begin{thm}\label{t6.7}
Let $\sfM$ be a smooth real 
submanifold of a complex manifold $\sfX$
and $\sfE$ a relatively closed subset of $\sfM,$ 
containing its $CR$ singularities. We assume that 
\begin{itemize}
 \item $\sfM$ is set-theoretically weakly $\qq$-pseudoconcave;
 \item the transversal complex dimension of $\sfE$ is less than $\qq$.
 \end{itemize}
Then the following are equivalent: 
\begin{enumerate}
  \item $\sfM\,{\backslash}\sfE$ has the holomorphic extension property;
  \item $\sfM$ has the holomorphic extension property;
 \item all germs of $CR$ functions 
 at points of $\sfM\,{\backslash}\sfE$ are $\Ci$-smooth;
 \item All germs in $\tiO_{\sfM}$ are $\Ci$-smooth.
\end{enumerate}
\end{thm} 
\begin{proof}
  Clearly  
 $(2)\,{\Rightarrow}\,(4)\,{\Rightarrow}\,(3).$ 
Since $(1)\,{\Leftrightarrow}\,(3)$  by \cite[Thm 1.1]{NP2015},
it suffices to check that $(1)\,{\Rightarrow}\,(2).$ This follows by applying the argument
in the proof of Thmeorem\,\ref{t6.1} to the transversal
intersection of $\sfM$ with complex submanifolds
$\sfY_{\!\pct}$ of dimension larger than $n{-}\qq,$ locally intersecting $\sfE$ at a
single point $\pct$ and transversal to $\sfM$ at $\pct.$ 
The intersection is in fact a smooth real submanifold and 
set-theoretically weakly $1$-pseudoconcave
by Lemma~\ref{l-1.8}. Let us point out the only additions needed here. 
We reduce as before to the case where 
$\sfM$ is a smooth real submanifold of an open subset of $\C^{n}$.
To show that $\cO_{\sfM,1}\,{=}\,\cO_{\sfM,0}$, it suffices to prove that,
for each point $\pct$ of $\sfE$ and open neighbourhood $\omegaup$ 
of $\pct$ in $\sfM$,
functions $f$ in $\cO_{\sfM,0}(\omegaup)$
have a holomorphic extension to a fixed neighbourhood
of $\pct$ in $\C^{n}$. Take a $\sfY_{\!\pct}$ of dimension larger than $n{-}\qq,$ locally intersecting $\sfE$ at 
the single point $\pct$ and transversal to $\sfM$ at $\pct.$ If $\omegaup'$ is an open neighbourhood
of $\pct$ in $\sfM$ with $\omegaup'\,{\Subset}\,\omegaup$, by a 
standard  open mapping theorem argument as above, there is $r\,{>}\,0$
such that, if  $f\,{\in}\,\cO_{\sfM,0}(\omegaup)$ and $\tilde{f}$ is its extension to an open neighbourhood
$\tilde{\omegaup}$ of $\omegaup$ in $\C^{n}$,
then the Taylor series of $\tilde{f},$ 
centred at points
$\zetaup\,{\in}\,\sfY_{\!\pct}\,{\cap}\,\partial\omegaup'$ 
converge in the ball $B(\zetaup,r)$. Since $\sfY_{\!\pct}\,{\cap}\,\sfM$ is set-theoretically weakly
pseudoconcave, by the maximum modulus principle also the Taylor series of $\tilde{f}$
centered at $\pct$ converges on $B(\pct,r)$. This completes the proof.
\end{proof}
%%%%%

\section{Some remarks on $CR$ singularities}\label{S7}
Real linear subspaces of $\C^{n}$
are $CR$ submanifolds. We 
consider the underlying real structure of $\C^{n}$, fixing an isomorphism $\C^{n}\,{\simeq}\,\R^{2n}$ 
and denoting by $\Jd\,{\in}\GL_{2n}(\R)$
the corresponding \emph{complex structure}.\par
Let $\sfE$ be a real linear subspace of $\C^{n}$. Its $CR$
dimension is the dimension of 
its complex subspace $\sfE\,{\cap}\,\Jd(\sfE)$. The sum 
$\sfE\,{+}\Jd(\sfE)$
is also a complex subspace and, by the intersection formula, 
\begin{equation}\label{e7.1}
 \dim_{\C}(\sfE\,{\cap}\,\Jd(\sfE))=\dim_{\R}(\sfE)\,{-}\,\dim_{\C}(\sfE\,{+}\,\Jd(\sfE)).
\end{equation}
If $T$ is a real $(2n){\times}d$ matrix, whose columns generate a $d$-dimensional \mbox{$\R$-linear}
subspace $\sfE$ of $\R^{2n}$, then the real dimension of $\,\sfE{+}\,\Jd(\sfE)\,$ is the rank of
the $(2n){\times}(2d)$ real matrix $(T,\Jd\,T)$. This rank is always ad even integer, greater than 
or equal to $d$. Let us indicate by $\sfE_{T}$ the real subspace of $\R^{2n}$ generated
by the columns of a rank $d$ 
matrix $T$  in $\R^{2n\times{d}}$. Then 
\begin{equation*} \crdim(\sfE_{T})
=d\,{-}\,\tfrac{1}{2}\rank(T,\Jd\,T),\;\;\crcodim(\sfE_{T})=\rank(T,\Jd\,T)\,{-}\,d.
\end{equation*}
\par 
In particular, when $2d\,{\geq}\,n$, the $d$-dimensional real subspaces which are not
$CR$-generic %$CR$ submanifolds 
are the $\sfE_{T}$'s for which the minor determinants
of rank $2n$ of $(T,\Jd\,T)$, obtained by adding $2n{-}d$ columns of $\Jd\,T$ to $T$, vanish. 
\par\medskip 
The grassmannian 
%Let 
$\Gr^{\R}_{d}(\C^{n})$ 
%be the grassmannian  
of $d$-dimensional real linear subspaces of $\C^{n}$ %.
% This 
is a $(d{\cdot}(2n{-}d))$-dimensional real
projective manifold. \par
For any pair of nonnegtive integers  
$(m,h)$ with $2m\,{+}\,h\,{=}\,d$, denote by $\Gr_{(m,h)}(\C^{n})$ the set of
real $d$-planes in $\Gr^{\R}_{d}(\C^{n})$ of $CR$-dimension $m$ and $CR$-codimension $h$.
This is a smooth real submanifold of $\Gr^{\R}_{d}(\C^{n})$ of real dimension \;\;
$
 2m(n\,{-}\,m)\,{+}\,h(2n\,{-}\,2m-h)=2(m\,{+}\,h)(n\,{-}\,m)\,{-}\,h^{2}.
$
\par\smallskip
Let $k$ be a positive integer with $2\,{\leq}\,k\,{\leq}n$ and take $d\,{=}\,2n{-}k$. 
Then 
\begin{equation*}
 \Gr_{2n-k}^{\R}(\C^{n})={\bigcup}_{1{\leq}h{\leq}k{\leq}2h}\Gr_{(n-h,2h-k)}(\C^{n})
\end{equation*}
is a real-analytic Whitney stratification of $\Gr^{\R}_{2n-k}(\C^{n})$.\par
The first term is a compact stratum, equal to $\Gr_{(n{-}((k{+}1)/2),1)}(\C^{n})$ for odd $k$ and to
$\Gr_{n{-}(k/2)}(\C^{n})\,{=}\,\Gr_{(n{-}(k/2),0)}(\C^{n})$
when $k$ is even, 
and the last one 
is its open dense 
submanifold $\Gr_{(n-k,k)}(\C^{n})$, 
consisting of $CR$-generic $(2n{-}k)$-dimensional real
planes of $\C^{n}$. For each $h$ we have 
\begin{equation} \begin{cases}
 \dim_{\R}(\Gr_{(n-h,2h-k)}(\C^{n}))=2h\,(n+k-h)-k^{2},\\
 \codim_{\R}(\Gr_{(n-h,2h-k)}(\C^{n}),\Gr^{\R}_{2n-k}(\C^{n}))=2(n-h)(k-h).
 \end{cases}
\end{equation}

 \par\smallskip

 For a smooth $(2n{-}k)$-dimensional real submanifold $\sfM$,
by using the (real) affine
 structure of $\C^{n}$, we may consider 
the Gauss map
\begin{equation}\label{e7.3} \tauup:
 \sfM\ni\pct\longrightarrow{T}_{p}\sfM\,{\in}\,\Gr_{2n-k}^{\R}(\C^{n}).
\end{equation} 
\begin{defn} The smooth real submanifold 
$\sfM$ of $\C^{n}$ has  \emph{$CR$ type
$(m,h)$} at its point $\pct$ if $\mathrm{T}_{\pct}\sfM\,{\in}\,\Gr_{(m,h)}(\C^{n})$. \par
 We introduce the notation 
 \begin{equation}
 \sfM_{(m,h)}=\{\pct\,{\in}\,\sfM\mid \mathrm{T}_{\pct}\sfM\,{\in}\,\Gr_{(m,h)}(\C^{n})\},
\end{equation}
 and  say that $\sfM$ has a
 \emph{transversal $CR$ singularity} at $\pct\,{\in}\,\sfM_{(m,h)}$ if $m{+}h\,{<}\,n$ and 
 its Gauss map \eqref{e7.3}
 is transversal to $\Gr_{(m,h)}(\C^{n})$ at $\pct$. 
\end{defn} 
\begin{rmk}\label{r7.1}
 The notions of $CR$ type and transversal $CR$ singularity are invariant by 
 local biholomorphisms and thus trivially extend to the case of general real smooth submanifolds
 of complex manifolds.
\end{rmk}
\begin{prop} Let $\sfM$ be a smooth $(2n{-}k)$-dimensional real submanifold of $\C^{n}$,
having transversal $CR$ singularities. 
Then 
\begin{itemize}
 \item $\sfM_{(n-k,k)}$ is open and dense in $\sfM$ and is a locally closed generic $CR$
 submanifold of $\C^{n}$;
 \item $\{\sfM_{(n-h,2h-k)}\}$ is a smooth stratification of $\sfM$,
 each $\sfM_{(n-h,2h-k)}$ being a smooth totally real submanifold of codimension 
 $2(n{-}h)(k{-}h)$ in $\sfM$.
\end{itemize}
\end{prop}

\begin{proof} Indeed, by the transversality condition, the codimension of
$\sfM_{(n-h,2h-k)}$ in $\sfM$ is the same as the codimension of $\Gr_{(n-h,2h-k)}(\C^{n})$
in $\Gr_{2n-k}^{\R}(\C^{n})$.
\end{proof} 
\begin{rmk}
In particular, if $\sfM$ has transversal $CR$ singularities, then either $\sfM_{(n-k+1,k-2)}$ is empty,
or has dimension 
 \begin{equation*} (2n{-}k)-2(n{-}k{+}1)
=k{-}2,
\end{equation*}
independent of $n$.\par
The integers $h$, with $h{\leq}k{\leq}2h$, for which a smooth $\sfM$ of real dimension $(2n{-}k)$
may carry a transversal $CR$ singularity of type $(n{-}h,2h{-}k)$ must satisfy 
\begin{equation*}\tag{$*$}
 2(n-h)(k-h)\leq{2}n-k.
\end{equation*}
If $n\,{=}\,k$, we obtain the condition that $2(n{-}h)^{2}\,{\leq}\,n$, i.e. that 
\begin{equation*}
 n-\sqrt{(n/2)}\,{\leq}\,\,h\,\,{\leq}\,\,n.
\end{equation*}
When $k\,{<}\,n$, condition $(*)$ is always satisfied if $k\,{\geq}\,2$ and $h\,{=}\,k{-}1$.
In general, for  $h\,{<}\,k\,{<}\,n$ we obtain 
\begin{equation*}
 k-h \leq \frac{2n-k}{2(n-h)}= \frac{n-k}{n-h}+\frac{k}{2(n-h)}<1+\frac{k}{2(n-k)},
\end{equation*}
showing that, when
$2n\,{>}\,3k{+}8$, 
there are at most transversal $CR$ singularities of type $(n{-}k{+}1,k{+}2)$; 
this observation justifies our special focusing  below on
$CR$-sin\-gu\-lar\-i\-ties
of this type. 
\end{rmk}
\begin{prop}\label{p7.3}
 Let $k,n$ be integers with $2{\leq}k{<}n$ and 
 $\sfM$ a $(2n{-}k)$-di\-men\-sional smooth real submanifold of $\C^{n},$
 whose nontransversal $CR$ singularities are isolated.
%having transversal $CR$ singularities.  
If $\sfM_{(n-k,k)}$ is weakly  
$\left[\tfrac{k+1}{2}\right]$-pseudoconcave and  has the holomorphic 
extension property, 
then $\sfM$ has the
holomorphic extension property at all points. 
 \end{prop} 
\begin{proof}
 Indeed, if $\sfE$ is the set of nontransversal $CR$ singularities of $\sfM$, then  
 the $CR$-singular locus  of $\sfM{\setminus}\sfE$ is contained in a locally finite union of smooth
 submanifolds of real dimension less or equal to $k{-}2$ and thus of transversal complex
 dimension $\left[\tfrac{k-1}{2}\right]$. The statement follows by Theorems\,\ref{t6.1} and \ref{t6.7}, after
 recalling that, by Proposition\,\ref{p1.7}, the notions of weak $\qq$-pseu\-do\-con\-cavity
 and set-theoretical weak $\qq$-pseudoconcavity coincide for smooth real submanifolds.
\end{proof}
%%%%%%%%%%
We note that, if in Proposition\,\ref{p7.3} 
we strengthen the assumption from weak to strict pseudoconcavity, 
then  $\sfM$ can only have transversal singularities
of type $(n{-}k{+}1,k{-}2)$. Indeed for the strict $\left[\tfrac{k+1}{2}\right]$-pseu\-do\-con\-cav\-ity 
of $\sfM_{(n{-}k,k)}$
one needs that $\left[\tfrac{k+1}{2}\right]\,{ \leq}\, \tfrac{n-k}{2}$. This implies that $n\,{\geq}\,2k$
and hence $\tfrac{k}{2(n{-}k)}\,{\leq}\,\tfrac{1}{2}$.\par\smallskip
By using invariance with respect to coordinate changes stated in Remark\,\ref{r7.1},
we get from Proposition\,\ref{p7.3}:
\begin{prop}
 Let $\sfM$ be a real submanifold of codimension $k$ of an $n$-dimensional complex manifold  $\sfX$
 and assume that the nontransversal $CR$ singularities of 
 $\sfM$ are isolated
 % has transversal $CR$ singularities 
 and that its
 $CR$-regular points form a generic $CR$ submanifold  $\sfM_{(n{-}k,k)}$
 of $\sfX$,
open and dense in $\sfM$, 
 weakly  $\left[\tfrac{k{+}1}{2}\right]$-pseu\-do\-con\-cave and having the holomorphic extension property.
 Then $\sfM$ has the holomorphic extension property. 
 \qed
\end{prop}
\par
\par\medskip 
 To study transversal \label{trasv}
$CR$ singularities of type $(n{-}k{+}1,k{-}2)$, we
 begin by describing an open neighbourhood of
an element of $\Gr_{(n{-}k+1,k{-}2)}(\C^{n})$. \par
Set for simplicity
$m\,{=}\,n{-}k{+}1$, $h\,{=}\,k-2$. Having fixed complex coordinates
$\zq^{1},\hdots,\zq^{m},\wq^{1},\hdots,\wq^{h},\tq$
in $\C^{n}$, we take 
the $(2n{-}k)$ plane 
\begin{equation*}
 \sfE_{0}=\{\wq^{j}{+}\bar{\wq}^{j}=0,\; 1{\leq}j{\leq}h,\;\tq\,{=}\,0\}\in\Gr_{(m,h)}(\C^{n}).
\end{equation*}
\par\noindent
It is  convenient to introduce 
\textit{real coordinates} 
$x,y\,{\in}\,\R^{n}$, $\uq,\vq\,{\in}\,\R^{h}$, $\rr,\sq\,{\in}\,\R$, with 
\begin{equation*}\label{rcoord}
 \zq\,{=}\,x\,{+}\,i\,y,\;\; (\wq^{1},\hdots,\wq^{h})\,{=}\,\uq\,{+}\,i\,\vq,\;\;
 \tq\,{=}\,\rr\,{+}\,i\,\sq.
\end{equation*}
The corresponding  complex structure on 
\smallskip\par \centerline{
$\R^{2n}\,{=}\,\R^{n}_{x}\,{\times}\,\R^{n}_{y}\,{\times}\,
\R^{h}_{\uq}\,{\times}\,\R^{h}_{\vq}\,{\times}\,\R^{1}_{\rr}\,{\times}\,\R^{1}_{\sq}$\;\;  is} 
\begin{equation*}
 \Jd = 
\begin{pmatrix}
 0 & {-}\Id_{m} & 0 & 0 & 0 & 0\\
 \Id_{m}& 0 & 0 & 0 & 0 & 0 \\
 0 & 0 & 0 & {-}\Id_{h} & 0 & 0\\
 0 & 0 & \Id_{h} & 0 & 0 & 0 \\
 0 & 0 & 0 & 0 & 0 & {-}1\;\\
 0 & 0 & 0 & 0 & 1 & 0
\end{pmatrix}.
\end{equation*}
We have $\sfE_{0}\,{=}\,\sfE_{T_{0}}$ with 
 
\begin{equation} \label{e7.4}
T_{0}=
\begin{pmatrix}
 \Id_{m}& 0 & 0\\
 0 & \Id_{m}& 0\\
 0 & 0 & \Id_{h}\\
 0 & 0 &0\\
 0 & 0 &0\\
  0 & 0 &0 
\end{pmatrix}.
\end{equation}
 We can represent a coordinate neighourhood of $T_{0}$ in $\Gr_{2m+h}^{\R}(\C^{n})$ by
 associating to matrices $V_{x},V_{y}\,{\in}\,\R^{h\times{m}}$, $V_{\vq}\,{\in}\,\R^{h\times{h}}$,
 $\Xi_{x},\Xi_{y},E_{x},E_{y}\,\,{\in}\R^{1\times{m}}$, and $\Xi_{\vq},E_{\vq}\,{\in}\,\R^{1\times{h}}$ 
 the $\R$-linear space generated by the columns of the matrix 
\begin{equation*} T=
\begin{pmatrix}
 \Id_{m}& 0 & 0\\
 0 & \Id_{m}& 0\\
 0 & 0 & \Id_{h}\\
 V_{x} & V_{y} &V_{\vq}\\
 \Xi_{x} & \Xi_{y}&\Xi_{\vq}\\
  E_{x} & E_{y} &E_{\vq} 
\end{pmatrix}.
\end{equation*} 
We get 
\begin{equation*}
 (T,\Jd{T})=\begin{pmatrix}
 \Id_{m}& 0 & 0 & 0 & {-}\Id_{m} & 0\\
 0 & \Id_{m}& 0 & \Id_{m} & 0 & 0\\
 0 & 0 & \Id_{h} & {-}V_{x}&{-}V_{y}&{-}V_{\vq}\\
 V_{x} & V_{y} &V_{\vq}& 0 & 0 &\Id_{h}\\
 \Xi_{x} & \Xi_{y}&\Xi_{\vq}&{-}E_{x}&{-}E_{y}&-E_{\vq}\\
  E_{x} & E_{y} &E_{\vq} &\Xi_{x}&\Xi_{y}&\Xi_{\vq}
\end{pmatrix}.
\end{equation*}\par
The matrix obtained by adding to $T$ the last $h$ columns of $\Jd\,T$ has rank $2m{+}2h$
in a neighborhood of  $0$ where $V_{\vq}$ has not the eigenvalue $i$. The rank of $(T,\,\Jd\,{T})$ is
equal to the rank of the matrix 
\begin{equation*} 
\begin{pmatrix}
 \Id_{m}& 0 & 0 & 0 & 0 & 0 \\
 0 & \Id_{m}& 0 & 0 & 0 & 0 \\
 0 & 0 & \Id_{h} & -V_{\vq} & {-}V_{y}&{-}V_{x}\\
 V_{x}&V_{y}&V_{\vq}&\Id_{h}&V_{x}&{-}V_{y}\\
 \Xi_{x}&\Xi_{y}&\Xi_{\vq}&{-}E_{\vq}&\Xi_{x}{-}E_{y}&{-}\Xi_{x}{-}E_{y}\\
 E_{x}&E_{y}&E_{\vq}& \Xi_{\vq} & \Xi_{y}{+}E_{x}& \Xi_{x}{-}E_{y}
\end{pmatrix},
\end{equation*}
i.e. to $2m$ plus the rank of the matrix 
\begin{equation*} 
\begin{pmatrix}
 \Id_{h} & -V_{\vq} & {-}V_{y}&{-}V_{x}\\
V_{\vq}&\Id_{h}&V_{x}&{-}V_{y}\\
\Xi_{\vq}&{-}E_{\vq}&\Xi_{x}{-}E_{y}&{-}\Xi_{x}{-}E_{y}\\
E_{\vq}& \Xi_{\vq} & \Xi_{y}{+}E_{x}& \Xi_{x}{-}E_{y}
\end{pmatrix}.
\end{equation*}
Thus, in the set where $V_{\vq}$ has not the eigenvalue $i$,
the rank of $(T,\,\Jd\,T)$ is  equal to  
$2m{+}2h$ plus the rank of the $2\,{\times}\,2h$ 
matrix 
\begin{equation*}\tag{$*$}
\begin{pmatrix}
 \Xi_{x}-E_{y} & {-}\Xi_{x}{-}E_{y}\\
 \Xi_{y}{+}E_{x}&\Xi_{x}{-}E_{y}
\end{pmatrix}+
\begin{pmatrix}
 \Xi_{\vq}&{-}E_{\vq}\\
 E_{\vq}&\Xi_{\vq}
\end{pmatrix} 
\begin{pmatrix}
 \Id_{h}&{-}V_{\vq}\\
 V_{\vq}&\Id_{h}
\end{pmatrix}^{-1}
\begin{pmatrix}
 V_{y}&V_{x}\\
 {-}V_{x}&V_{y}
\end{pmatrix}.
\end{equation*}
Since we know that $(T,\,\Jd\,T)$ has even rank, this matrix is either $0$ or has rank~$2$.
We note that the second summand in $(*)$ vanishes to the second order at $0$.
Thus the tangent space to $\Gr_{(m,h)}(\C^{n})$ at $0$ is the space where
\begin{equation*} 
\begin{cases}
 E_{y}=\Xi_{x},\\
 E_{x}={-}\Xi_{y},
\end{cases}
\end{equation*}
i.e. 
\begin{equation}\label{e6.4}
 \mathrm{T}_{\sfE_{0}}\Gr_{(m,h)}(\C^{n})\simeq \left.\left\{
 \begin{pmatrix}
  V_{x} & V_{y} &V_{\vq}\\
 \Xi_{x} & \Xi_{y}&\Xi_{\vq}\\
  {-}\Xi_{y} & \Xi_{x} &E_{\vq} 
\end{pmatrix}\right| \, \begin{gathered}
 V_{x},V_{y}\,{\in}\,\R^{h\times{m}},\, V_{\vq}\,{\in}\,\R^{h\times{h}},\\ \Xi_{x},\Xi_{y}\,{\in}\R^{1\times{m}},\,
 \Xi_{\vq},E_{\vq}\,{\in}\,\R^{1\times{h}}
 \end{gathered}\right\}.
\end{equation}
\par\medskip
Consider, for hermitian symmetric $m\,{\times}\,m$ matrices
$H_{1},\hdots,H_{h}$,  complex matrices $B,C\,{\in}\,\C^{m\times{m}}$
and $D\,{\in}\,\C^{m\times{h}}$, with $C$ symmetric,
\begin{equation}\label{e7.7}
 \sfM\,:\, 
\begin{cases}
 \wq^{1}\,{+}\,\bar{\wq}^{1}=\zq^{*}H_{1}\zq,\\
 \hdots\hdots\hdots\hdots\\
  \wq^{h}\,{+}\,\bar{\wq}^{h}=\zq^{*}H_{h}\zq,\\
 \tq=\zq^{*}B\zq+\tfrac{1}{2}\zq^{*}C\,\bar{\zq}+\zq^{*}D\vq.
\end{cases}
\end{equation}
The \textit{quadric} $\sfM$ 
is a smooth real submanifold of dimension $(2n{-}k)$ of $\C^{n}$.
For $B\,{=}\,0$, $C\,{=}\,0$ and $D\,{=}\,0$, 
$\sfM$ is a generic $CR$ submanifold of type $(m,h)$ of the hyperplane
$\{\tq\,{=}\,0\}$.\par 
 If $B$, $C$, $D$ are not all $0$, 
 then $\sfM$ is at general points a generic smooth $CR$ submanifold
of type $(m{-}1,k)$ of $\C^{n}$, its  
$CR$ singularities occurring at points where the space of holomorphic differentials of real functions
vanishing on $\sfM$ have 
rank
 less than $k$, i.e. on its subset $\sfM^{\sharp}$
where
\begin{equation}\label{e7.8a}
\left\{
\begin{aligned}
 d\wq^{1}{-}\zq^{*}H_{1}d\zq,\;\;\hdots,\;d\wq^{h}{-}\zq^{*}H_{h}dz,\;\; d\tq{-}\zq^{*}B\,d\zq{-}\tfrac{i}{2}\zq^{*}D\,d\wq,\\
 (\zq^{*}B^{*}\,{+}\,\zq^{\intercal}C^{*}\,{+}\,\vq^{\intercal}{D}^{*})d\zq+\tfrac{i}{2}z^{\intercal}\bar{D}\,d\wq
\end{aligned}\right.
\end{equation}
are linearly dependent. Set $D\,{=}\,(D_{i,\alpha})_{
\begin{smallmatrix}
 i=1,\hdots,m,\\
 \alpha=1,\hdots,h
\end{smallmatrix}
}$, $H_{\alpha}=(H_{\alpha\;i,j})_{
\begin{smallmatrix}
 i,j=1,\hdots,m,\\
 \alpha=1,\hdots,h\;\;\;
\end{smallmatrix}}$.
Modulo the first $h$ differentials, we have 
\begin{align*}
 \zq^{\intercal}\bar{D}\,d\wq \,&{=}\,{\sum}_{i=1}^{m}{\sum}_{\alpha=1}^{h} \zq^{i}\bar{D}_{i,\alpha}d\wq^{\alpha}\\
 &={\sum}_{i,j,s=1}^{m}{\sum}_{\alpha=1}^{h} \zq^{i}\bar{D}_{i,\alpha} \bar{\zq}^{j}H_{\alpha;j,s}d\zq^{s}.
\end{align*}
We may consider 
$\bar{D}H\,{\coloneqq}\,\left({\sum}_{\alpha=1}^{h}\bar{D}_{i,\alpha}H_{\alpha;j,s}\right)_{i,j,s=1,\hdots,m}$ 
as a tensor two-covariant and one-contravariant, so that its adjoint  defines  a bilinear
map 
\begin{equation*}
 HD^{\intercal}:\C^{m}\times\C^{m}\to\,\C^{m}.
\end{equation*}
The $h{+}2$ differentials of \eqref{e7.8a} are linerarly dependent iff
\begin{equation*}
 \zq^{*}B^{*}+\zq^{\intercal}{C}^{*}+\vq^{\intercal}D^{*}+\tfrac{i}{2} \zq^{\intercal}\bar{D}(\zq^{*}H)=0.
\end{equation*}
With the notation introduced above, 
by taking the adjoint this becomes 
\begin{equation}\label{e7.8}
 \Lambda_{B,C,D}(\zq,\vq)+\tfrac{i}{2}(H\,D^{\intercal})(\bar{\zq},\zq)=0,\end{equation}
where  
$ \Lambda_{B,C,D}$ 
the $\R$-linear map 
\begin{equation}\label{e7.9}
 \Lambda_{B,C,D}:\C^{m}\times\R^{h}\ni(\zq,\vq)\to B\zq+C\bar{\zq}+D\vq\in\C^{m},
\end{equation}
yielding 
\begin{equation}\label{e7.10}
 \sfM^{\sharp}=\{(\zq,\wq,\tq)\,{\in}\,\sfM\,{\mid}\,  
 \Lambda_{B,C,D}(\zq,\vq)+\tfrac{i}{2} %(\zq^{*}H\zq)^{\intercal}D^{\intercal}\bar{\zq}=0\}.
 (HD^{\intercal})(\bar{\zq},\zq)=0\}.
\end{equation}\par\smallskip
Let us consider first the simpler case where $D\,{=}\,0$.
Any  $\R$-linear endomorphism  of $\C^{m}$ has a unique decomposition as a sum of
its $\C$-linear and its anti-$\C$-linear components. 
Using the real coordinates of page\,\pageref{rcoord}  
we set 
$B\,{=}\,F\,{+}\,iG$,
$C\,{=}\,P\,{+}\,iQ$, with $F,G,P,Q$ real $m\,{\times}\,m$ matrices, $P,Q$ symmetric.
The $\R$-linear map 
\begin{equation}\label{e7.11}
\Lambda_{B,C}:
 \C^{m}\ni\zq \longrightarrow B\,\zq + C\,\bar{\zq}\in\C^{m}
\end{equation}
has matrix representations  
\begin{equation*} 
\begin{pmatrix}
 F+P & Q-G\\
 G+Q & F-P
\end{pmatrix} 
\begin{pmatrix}
 x \\ y
\end{pmatrix},\quad\text{or}\quad 
\begin{pmatrix}
 B & C \\
 \bar{C}&\bar{B} 
\end{pmatrix} 
\begin{pmatrix}
 \zq \\ \bar{\zq}
\end{pmatrix},
\end{equation*}
the rank of the real matrix on the left being half the rank of the complex matrix on the right. Indeed,
if $\left(\begin{smallmatrix} x \\ y
\end{smallmatrix}\right)$ belongs to the kernel of $\left( 
\begin{smallmatrix}
 F{+}P & Q{-}G\\
 G{+}Q&F{-}P
\end{smallmatrix}\right)$, then $\left(\begin{smallmatrix} x{+}iy \\ x{-}i y
\end{smallmatrix}\right)$ belongs to the kernel of 
$\left( 
\begin{smallmatrix}
B & C\\
\bar{C}&\bar{B}
\end{smallmatrix}\right)$. On the other hand, the kernel of~$\left( 
\begin{smallmatrix}
B & C\\
\bar{C}&\bar{B}
\end{smallmatrix}\right)$ is invariant under the 
involution 
$\left(\begin{smallmatrix} \zq \\ \wq
\end{smallmatrix}\right)\,{\mapsto}\,\left(\begin{smallmatrix} \bar{\wq} \\ \bar{\zq}
\end{smallmatrix}\right)$ and the corresponding 
eigendecomposition of $\C^{2m}$ yields a direct sum decomposition of
this kernel into two spaces having equal dimensions.
\par
Being an $\R$-linear subspace, $\ker(\Lambda_{B,C})$ is a $CR$ submanifold of $\C^{m}$.
Let $(r_{B,C},s_{B,C})$ be its $CR$-type: we have \begin{equation*}
r_{B,C}\,{=}\,\dim_{\C}(\ker(B)\,{\cap}\,\ker(\bar{C})), \qquad
 2r_{B,C}\,{+}\,s_{B,C}\,{=}\,\dim_{\R}\ker(\Lambda_{B,C}).
 \end{equation*}
\par In this case
the set $\sfM^{\sharp}$ 
is a smooth $CR$ submanifold of type $(r_{B,C},s_{B,C}{+}h)$ 
 of $\sfM$. It is totally real when $r_{B,C}\,{=}\,0$.
In particular, 
When $\Lambda_{B,C}$ is nonsingular, it is an $h$-dimensional real linear subspace.
In general, we get:
\begin{lem}
 If $\Lambda_{B,C,D}$ has real rank $2m$, then there is an open neighbourhood $U$ of $0$ in $\C^{n}$
 such that 
 $\sfM^{\,\sharp}\,{\cap}\,U$ is a smooth real $h$-dimensional submanifold of
 $\sfM\,{\cap}\,U$.
\end{lem} 
\begin{proof}
This follows from the implicit function theorem applied to the map 
\begin{equation*}
 \C^{m}\times\R^{h}\ni(\zq,\vq)\to \Lambda_{B,C,D}(\zq,\vq)+\tfrac{i}{2}(\zq^{*}H\zq)^{\intercal}D^{\intercal}\bar{\zq}
 \in\C^{m}.
\end{equation*}
Indeed, if its Jacobian at $0$ has rank $2m$, then its zero set is, near the origin,
a smooth $h$-dimensional real submanifold.
\end{proof}
We prove, under this rank condition, a transversality result 
for the Gauss map \eqref{e7.3}, slightly generalizing some   propositions in  \cite{Cof2009,Gar}.
\par\smallskip 

Using the real coordinates of page\,\pageref{rcoord},
we have for the $\sfM$  of \eqref{e7.7} the parametrical description 
\begin{equation*}\tag{$**$}
\begin{cases}
 \uq^{j}=\tfrac{1}{2}z^{*}H_{j}z, \qquad 1{\leq}j{\leq}h,\\
 \rr=\tfrac{1}{2}z^{*}(B\,{+}\,B^{*})z+\tfrac{1}{2}\re(\zq^{*}C\,\bar{\zq})+\re(\zq^{*}D\vq),\\
 \sq=\tfrac{i}{2}z^{*}(B^{*}{-}\,B)z+\tfrac{1}{2}\im(\zq^{*}C\,\bar{\zq})+\im(\zq^{*}D\vq).
 \end{cases}
\end{equation*}
\par
The tangent space to $\sfM$ at $0$ is generated by the columns of the matrix $T_{0}$
of \eqref{e7.4}.
We write $H_{j}\,{=}\,K_{j}\,{+}i\,L_{j}$, with $K_{j}$ real symmetric and $L_{j}$ real antisymmetric.
Then the first lines in $(**)$ can be rewritten in terms of the real coordinates by 
\begin{equation*}
 2\uq^{j}=x^{\intercal}K_{j}x+y^{\intercal}K_{j}y+y^{\intercal}L_{j}x-y^{\intercal}L_{j}x,\;\; 1{\leq}j{\leq}h.
\end{equation*}
\par
Let $D\,{=}\, R\,{+}\,i\,S$, with $R,S\,{\in}\,\R^{m\times{h}}$.
The last two lines are, in terms of real coordinates, 
\begin{equation*} 
\begin{cases}
 \rr=x^{\intercal}(F{+}\tfrac{1}{2}P)x \,{+}\, y^{\intercal}(F{-}\tfrac{1}{2}P)y\,{+}\,x^{\intercal}(\tfrac{1}{2}Q-G)y\,{+}\,y^{\intercal}(G{+}\tfrac{1}{2}Q)x
 \,{+}\,x^{\intercal}R\vq\,{+}\,y^{\intercal}S\vq,\\
 \sq=x^{\intercal}(G{+}\tfrac{1}{2}Q)x\,{+}\,y^{\intercal}(G{-}\tfrac{1}{2}Q)y\,{+}\, x^{\intercal}(F{-}\tfrac{1}{2}P)y\,{-}\,y^{\intercal}(F{+}\tfrac{1}{2}P)x
 \,{+}\,x^{\intercal}S\vq
\,{-}\,y^{\intercal}{R}\vq.\end{cases}
\end{equation*}We use below 
\begin{equation*}
 K= 
\begin{pmatrix}
 K_{1}\\ \vdots \\ K_{h}
\end{pmatrix},\;\;\; L= 
\begin{pmatrix}
 L_{1}\\ \vdots \\ L_{h}
\end{pmatrix}\end{equation*}
The Gauss map $\tauup(x,y,\vq)$ of 
\eqref{e7.3} has the parametric expression 
{\footnotesize
\begin{equation*}
\begin{pmatrix}
 \Id_{m}&0&0\\
 0 &\Id_{m}&0\\
 0&0&\Id_{h}\\
  x^{\intercal}K+y^{\intercal}L & y^{\intercal}K-x^{\intercal}L & 0\\
x^{\intercal}(F{+}F^{\intercal}{+}P)\,{+}\,y^{\intercal}(G{-}G^{\intercal}{+}Q)\,{+}\,\vq^{\intercal}R^{\intercal} &
y^{\intercal}(F{+}F^{\intercal}{-}P)\,{+}\,x^{\intercal}(G^{\intercal}{-}G{+}Q)\,{+}\,\vq^{\intercal}S^{\intercal} &
x^{\intercal}{R}{+}y^{\intercal}{S}\\
x^{\intercal}(G{+}G^{\intercal}{+}Q)\,{+}\,y^{\intercal}(F^{\intercal}{-}F{-}P)\,{+}\,\vq^{\intercal}S^{\intercal} &
y^{\intercal}(G{+}G^{\intercal}{-}Q)\,{+}\,x^{\intercal}(F{-}F^{\intercal}{-}P)\,{-}\,\vq^{\intercal}R^{\intercal} &
x^{\intercal} S{-}y^{\intercal}R
\end{pmatrix}.
\end{equation*}}
 Being linear in the parameters $x$, $y$, $\vq$,
 its tangent at $0$ is, in the local coordinates
 corresponding to the last $k$ lines of the matrix above: 
 {\footnotesize  \
\begin{equation*} 
\begin{pmatrix}
   x^{\intercal}K+y^{\intercal}L & y^{\intercal}K-x^{\intercal}L & 0\\
x^{\intercal}(F{+}F^{\intercal}{+}P)\,{+}\,y^{\intercal}(G{-}G^{\intercal}{+}Q)\,{+}\,\vq^{\intercal}R^{\intercal} &
y^{\intercal}(F{+}F^{\intercal}{-}P)\,{+}\,x^{\intercal}(G^{\intercal}{-}G{+}Q)\,{+}\,\vq^{\intercal}S^{\intercal} &
x^{\intercal}{R}{+}y^{\intercal}{S}\\
x^{\intercal}(G{+}G^{\intercal}{+}Q)\,{+}\,y^{\intercal}(F^{\intercal}{-}F{-}P)\,{+}\,\vq^{\intercal}S^{\intercal} &
y^{\intercal}(G{+}G^{\intercal}{-}Q)\,{+}\,x^{\intercal}(F{-}F^{\intercal}{-}P)\,{-}\,\vq^{\intercal}R^{\intercal} &
x^{\intercal} S{-}y^{\intercal}R
\end{pmatrix}.
\end{equation*}
}
\par
In the description \eqref{e6.4} of $\mathrm{T}_{T_{0}}\Gr_{(m,h)}(\C^{n})$,
$V_{x},V_{y},V_{\vq},\Xi_{\vq},E_{\vq}$ can be arbitrarily chosen.
Thus the condition for the Gauss map of $\sfM$ to be transversal to $\Gr_{(m.h)}(\C^{n})$ at $0$ 
reads off as the fact that $\R^{2\times(2m)}$ is sum of the subspaces {\footnotesize
\begin{align*}
&\left.\left\{ 
\begin{pmatrix}
 \Xi_{x} & \Xi_{y}\\
 {-}\Xi_{y}&\Xi_{x}
\end{pmatrix}\right| \Xi_{x},\Xi_{y}\,{\in}\,\R^{1\times{m}}\right\}\qquad\text{and}\\
&  
\left.\left\{
\begin{pmatrix}
x^{\intercal}(F{+}F^{\intercal}{+}P)\,{+}\,y^{\intercal}(G{-}G^{\intercal}{+}Q)\,{+}\,\vq^{\intercal}R^{\intercal} &
y^{\intercal}(F{+}F^{\intercal}{-}P)\,{+}\,x^{\intercal}(G^{\intercal}{-}G{+}Q)\,{+}\,\vq^{\intercal}S^{\intercal} \\[4pt]
x^{\intercal}(G{+}G^{\intercal}{+}Q)\,{+}\,y^{\intercal}(F^{\intercal}{-}F{-}P)\,{+}\,\vq^{\intercal}S^{\intercal} &
y^{\intercal}(G{+}G^{\intercal}{-}Q)\,{+}\,x^{\intercal}(F{-}F^{\intercal}{-}P)\,{-}\,\vq^{\intercal}R^{\intercal} 
\end{pmatrix}
\right| \begin{gathered}
x,y\,{\in}\,\R^{n}\\[-5pt]
\vq\,{\in}\,\R^{h}\end{gathered}\right\}.
\end{align*}}
We can sum to all elements of the second subspace  an element of the first one to obtain a matrix
with zero second row. Thus the condition for transversality is that 
{\footnotesize
\begin{equation*}
 \{(x^{\intercal}(F^{\intercal}{+}P)\,{+}\,y^{\intercal}(Q{-}G^{\intercal})\,{+}\,\vq^{\intercal}R^{\intercal},\;
 x^{\intercal}(G^{\intercal}{+}Q)\,{+}\,y^{\intercal}(F^{\intercal}{-}P)\,{+}\,\vq^{\intercal}S^{\intercal}\mid
 x,y\,{\in}\,\R^{n},\,\vq\,{\in}\,\R^{h}\}\,{=}\,\R^{1\times(2m)}.
\end{equation*}
}
Thus the transversality condition is that 
the real matrix 
\begin{equation} \label{e7.12} [\Lambda_{B,C,D}]=
\begin{pmatrix}
 F+P & Q{-}G & R\\
 G+Q & F-P & S
\end{pmatrix}
\end{equation}
has rank $2m$.
%Its rank 
%is greater or equal to 
%that of the $\R$-linear map $\Lambda_{B,C}$. 
We conclude with
\begin{prop} \label{p7.6}
Let $\sfM$ be the quadric of \eqref{e7.7}.
A necessary and sufficient condition in order that the Gauss map \eqref{e7.3}
be transversal to $\Gr_{(m,h)}(\C^{n})$ at $0$ is that the matrix $[\Lambda_{B,C,D}]$ of 
\eqref{e7.12}
has rank $2m$. \qed
\end{prop} 
\begin{rmk}
 This holds true, in particular, when
$\Lambda_{B,C}$ is nonsingular.
\par
 When $h\,{=}\,0$, the condition reduces to the fact that the $\R$-linear map $\Lambda_{B,C}$ of 
 \eqref{e7.11} is nonsingular and
 we recover the transversality criterion of \cite{Cof2009,Gar}.
\end{rmk}
\begin{exam}
Let $\zq,\wq,\tq$ be holomorphic coordinates in 
 $\C^{3}$ and consider, for $\lambdaup\,{\in}\,\C$, 
\begin{equation*} \sfM_{\lambdaup}\,{:}\;
\begin{cases}
 \wq+\bar{\wq}=0,\\
 \tq=\zq\bar{\zq}\,{+}\,\tfrac{1}{2}\bar{\zq}^{2}\,{+}\,\lambdaup\,{\wq}\bar{\zq}.
\end{cases}
\end{equation*}
The $CR$ singularities of $\sfM_{\lambdaup}$ occur at points where the differentials 
\begin{equation*}
 d\wq,\;\;\; d\tq\,{-}\,\bar{\zq}\,d\zq \,{-}\,\lambdaup\,\bar{\zq}\,d\wq,\;\;\; 
 (\zq\,{+}\,\bar{\zq}\,{+}\bar{\,\lambdaup}\,\bar{\wq})d\zq
\end{equation*}
are linearly dependent. Thus we obtain that 
\begin{align*}
\sfM_{\lambdaup}^{\sharp}&={\sfM_{\lambdaup}}_{(1,1)}\,{=}\,\{(\zq,\wq,\tq)\in\sfM_{\lambdaup}\mid
\zq{+}\bar{\zq}{-}i\bar{\,\lambdaup}\,\im(\wq)\,{=}\,0\}
%\\
%&\,{=}\, 
%\begin{cases}
% \{(i\tauup,0,{-}\tfrac{1}{2}\tauup^{2})\mid \tauup\,{\in}\,\R\}, &\re(\lambdaup)\,{\neq}\,{0},\\
% \{\tfrac{i}{2}\,\lambdaup\,\wq\,{+}\,iy, \, \wq,
% {-}\tfrac{1}{8}(\lambdaup\wq+y)^{2})\mid \re(\wq)=0,\, y\,{\in}\,\R\}, & \re(\lambdaup)\,{=}\,0.
%\end{cases}
\end{align*}
is a smooth submanifold of real dimension $1$ if $\re(\lambdaup)\,{\neq}\,0$ and $2$ when
$\re(\lambdaup)\,{=}\,0$. 
The associated matrix \eqref{e7.12} is 
\begin{equation*} 
\begin{pmatrix}
 1 & 0 & {-}\im(\lambdaup)\\
 0 & 0 & \; \re(\lambdaup)
\end{pmatrix}
\end{equation*}
and, according to Proposition\,\ref{p7.6}, the $CR$ singularity at $0$ is transversal iff
$\re(\lambdaup)\,{\neq}\,0.$ 
\par\smallskip
Consider next 
\begin{equation*} \sfM'\,{:}\;
\begin{cases}
 \wq+\bar{\wq}=0,\\
 \tq={\wq}\bar{\zq}.
\end{cases}
\end{equation*}
The $CR$ singularities of $\sfM'$ occur at points where the differentials 
\begin{equation*}
 d\wq,\; d\tq\,{-}\,\bar{z}\,d\wq,\; \bar{\wq}\,d\zq
\end{equation*}
are linearly dependent: we obtain 
\begin{equation*}
{\sfM'}^{\sharp}=\sfM'_{(1,1)}\,{=}\,\{\wq\,{=}\,0,\;\tq\,{=}\,0\}.
\end{equation*}The associated matrix \eqref{e7.12} is, in this case, 
$\left(  \begin{smallmatrix}
 0 & 0 & i \\
 0 & 0 & 0
\end{smallmatrix}\right)$: the $CR$ singularity at $0$ is not transversal and in fact 
%\begin{equation*} 
% \begin{pmatrix}
% 0 & 0 & i \\
% 0 & 0 & 0
%\end{pmatrix}
% \end{equation*}
%and 
the real dimension ($2$) of ${\sfM'}^{\sharp}$ is larger than in the transversal
case.
\end{exam} 
\begin{exam}
 Let $\zq,\wq^{1},\wq^{2},\tq$ be holomorphic coordinates in $\C^{4}$. Let
 $\uq^{j}\,{=}\,\re(\wq^{j}),$ $\vq^{j}\,{=}\,\im(\wq^{j})$ ($j\,{=}\,1,2$) and 
\begin{equation*}
 \sfM\,:\, 
\begin{cases}
 \uq^{1}=0,\\
 \uq^{2}=0,\\
 \tq=(\lambdaup_{1}\vq^{1}{+}\lambdaup_{2}\vq^{2})\bar{\zq}
 ={-}i(\lambdaup_{1}\wq^{1}{+}\lambdaup_{2}\wq^{2})\bar{\zq}.
\end{cases}
\end{equation*} 
This is a smooth submanifold of real codimension $4$ of $\C^{4}$.
Its $CR$ singularities occur at points where 
\begin{gather*}
 d\wq^{1},\;
 d\wq^{2},\;
 d\tq+i\,\bar{\zq}(\lambdaup_{1}d\wq^{1}{+}\lambdaup_{2}d\wq^{2}),\;\;
 (\bar{\,\lambdaup}_{1}\bar{\wq}^{1}{+}\bar{\,\lambdaup}_{2}\bar{\wq}^{2})d\zq,
\end{gather*}
are linearly dependent.  Hence 
\begin{equation*}
 \sfM^{\sharp}\,{=}\,\sfM_{1,2}=\{(\zq,\wq,\tq)\,{\in}\,\sfM\mid \lambdaup_{1}\vq_{1}{+}\lambdaup_{2}\vq_{2}\,{=}\,0\}. 
\end{equation*}
When $\lambdaup_{1},\lambdaup_{2}$ are linearly independent over $\R$, the $CR$ singularities are
transversal and $\sfM\,{=}\,\{\wq\,{=}\,0,\, \tq\,{=}\,0\}$ is a complex line (real dimension two).
When $\lambdaup_{1},\lambdaup_{2}$ generate a real line, then 
the $CR$ singularity is not transversal and in fact 
$\sfM^{\sharp}$ is a $3$-dimensional
real  subspace of type $(1,1)$ of $\C^{4}$.
\end{exam}
\par\smallskip
By the implicit function theorem and some easy change of holomorphic coordinates  we obtain 
\begin{lem}\label{l7.9}
 Let $k,n,p$ be integers with $2{\leq}k{\leq}n$, $0{\leq}2p{\leq}k$ 
 and $\sfM$ a smooth $(2n{-}k)$-dimensional real smooth
 submanifold of an $n$-dimensional complex manifold $\sfX$. 
 Assume that $\pct\,{\in}\,\sfM$ and 
 $\mathrm{T}_{\pct}\sfM\,{\in}\,\Gr_{(n-k+p,k-2p)}(\mathrm{T}_{\pct}\sfX)$.
 Set $h\,{=}\,k{-}2p$,
 $m\,{=}\,n{-}h{-}p$. Then we can choose real coordinates $x,y\,{\in}\,\R^{m}$, $\uq,\vq\,{\in}\,\R^{h}$,
 $\rr,\sq\,{\in}\,\R^{p}$ such that 
\begin{equation*}
 \zq\,{=}\,x\,{+}\,i\,y\,{\in}\,\C^{m},\;\; \wq\,{=}\,\uq\,{+}\,i\,\vq\,{\in}\,\C^{h},\;\; \tq\,{=}\,\rr\,{+}\,i\,\sq\,{\in}\,\C^{p}
\end{equation*}
are holomorphic coordinates centred at $\pct$ and locally
\begin{equation}\label{e7.13}
 M\,{:}\, 
\begin{cases}
 \uq^{j}=\tfrac{1}{2}\zq^{*}H_{j}\zq + 0(3),& 1\leq{j}\leq{h},\\
 \tq^{\iota}=\zq^{*}B_{\iota}\zq+\tfrac{1}{2}\zq^{*}C_{\iota}\bar{\zq}+\zq^{*}D_{\iota}\vq + 0(3), & 1\leq{\iota}\leq{p},
\end{cases}
\end{equation}\label{e7.14}
where $H_{j},B_{\iota},C_{\iota}\,{\in}\,\C^{m\times{m}},$ $D_{\iota}\,{\in}\,\C^{m\times{h}}$,
with $C_{\iota}$ symmetric and $H_{j}$ hermitian symmetric.\qed
\end{lem}
If $\sfM$ is described by \eqref{e7.13}, we set 
\begin{equation}\label{e7.15}
 H_{h+2\iota-1}= \tfrac{1}{2}(B_{\iota}\,{+}\,B^{*}_{\iota}),\;\; H_{h+2\iota}=\tfrac{i}{2}(B_{\iota}^{*}{-}B_{\iota}),\;\;
 1{\leq}\iota{\leq}p.
\end{equation}
and, for $\vec{c}\,{=}\,(c^{1},\hdots,c^{k})\,{\in}\,\R^{k},$
\begin{equation}
 H_{\vec{c}}={\sum}_{j=1}^{k}c^{j}H_{j},\;\;\;\text{for}\,\; \vec{c}\,{=}\,(c^{1},\hdots,c^{k})\in\R^{k}.
\end{equation}
\begin{rmk} Using the notation of \eqref{e7.9} and \eqref{e7.14}, \eqref{e7.15},
we note that 
\begin{itemize}
 \item the ranks of the maps in the 
 $\C$-linear span of $\Lambda_{B_{1},C_{1},D_{1}},\hdots
 \Lambda_{B_{p},C_{p},D_{p}}$,
 \item the signatures of the hermitian symmetric 
 matrices in the $\R$-linear span of  $H_{1},\hdots,H_{k}$ 
\end{itemize}
are biholomorphic invariants of $\sfM$ at $\pct$.
\end{rmk}
\begin{exam} We consider next examples with $p\,{=}\,2$ and $h\,{=}\,0$.
Let $m\,{\geq}2$ 
be a positive integer and consider, for $m\,{\times}\,m$ complex matrices
$B_{1},B_{2}$, 
the  \textit{quadric} of 
$\C^{2+m}\,{=}\,\C^{2}_{\wq}\,{\times}\,\C^{m}_{\zq}$
\begin{equation*}
 \sfM\,:\, 
\begin{cases}
 \wq^{1}=\zq^{*}B_{1}\zq,\\
 \wq^{2}=\zq^{*}B_{2}\zq.
\end{cases}
\end{equation*}
This is a $2m$-dimensional smooth real submanifold of $\C^{2+m}$.
\par 
If $B_{1}$ and $B_{2}$ are linearly independent, then
 the $CR$-regular points form a $CR$ manifold
$\sfM_{m-2,4}$ which is open and dense in $\sfM$  and  
\begin{equation*}
 \sfM=\sfM_{(m-2,4)}\cup\sfM_{(m-1,2)}\cup\sfM_{m,0},
\end{equation*}
its $CR$ singularities occurring at points where the four differentials 
\begin{equation*}
 d\wq^{1}-\zq^{*}B_{1}d\zq,\;\; d\wq^{2}-\zq^{*}B_{2}d\zq,\;\; z^{*}B_{1}^{*}d\zq,\;\; \zq^{*}B_{2}^{*}d\zq
\end{equation*}
are linearly dependent:  $\sfM^{\sharp}\,{=}\, \sfM_{(m-1,2)}\,{\cup}\,\sfM_{m,0}$, with 
\begin{equation*} 
\begin{cases}
 \sfM_{(m,0)}=\{(\wq,\zq)\,{\in}\,\sfM\mid B_{1}\zq\,{=}\,0,\; B_{2}\zq\,{=}\,0\},\\
 \sfM_{(m-1,2)}=\{(\wq,\zq)\,{\in}\,\sfM\,{\setminus}\sfM_{(m,0)}\mid 
 \exists\,(\lambdaup^{1},\lambdaup^{2})\,{\in}\,\C^{2}{\setminus}\{0\} \;
 \text{s.t.}\; \lambdaup^{1}B_{1}\zq\,{+}\,\lambdaup^{2}B_{2}\zq\,{=}\,0\}.
  \end{cases}
\end{equation*}
If $\ker(B_{1})\,{\cap}\,\ker(B_{2})\,{=}\,\{0\}$, then $\sfM_{(m,0)}\,{=}\,\{0\}$. 
Let us assume, in addition, that $B_{1}$ is nonsingular. 
The points $(\wq,\zq)$ of $\sfM_{(m-1,2)}$ are those for which $\zq$ is a nonzero
eigenvector of $B_{1}^{-1}B_{2}$. If the geometric multiplicity of all eigenvalues
of $B_{1}^{-1}B_{2}$ is one, then $\sfM_{(m-1,2)}$ is a smooth totally real submanifold
of dimension $2$ of $\sfM$. \par\smallskip 
Let us consider a specific \textit{singular} example, with $m\,{=}\,3$ and 
\begin{equation*}
 B_{1}= 
\begin{pmatrix}
 1 & 0 & 0\\
 0 & 0 & 1\\
 0 & 0 & 0
\end{pmatrix},\;\;  B_{2}= 
\begin{pmatrix}
 0 & 1 & 0\\
 0 & 0 & 0\\
 0 & 0 & 1
\end{pmatrix}.
\end{equation*}
In this case we obtain 
\begin{equation*}
 \big(\lambdaup_{1}B_{1}+\lambdaup_{2}B_{2} \big)
\begin{pmatrix}
 \lambdaup_{2}\\ {-}\lambdaup_{1}\;\,\\ 0
\end{pmatrix}=0,\;\;\;\forall\lambdaup_{1},\lambdaup_{2}\,{\in}\,\C,
\end{equation*}
and therefore $\sfM_{(2,2)}\,{=}\,\{(\wq,\zq)\,{\in}\,\sfM\,{\mid}\,\zq^{3}{=}\,0\}$ is a smooth submanifold
of $\sfM$ of real dimension $4$.
\par\medskip
Let us consider next the more general \textit{quadric} of 
$\C^{2+m}\,{=}\,\C^{2}_{\wq}\,{\times}\,\C^{m}_{\zq}$
\begin{equation*}
 \sfM\,:\, 
\begin{cases}
 \wq^{1}=\zq^{*}B_{1}\zq+\tfrac{1}{2}\zq^{*}C_{1}\bar{\zq},\\
 \wq^{2}=\zq^{*}B_{2}\zq+\tfrac{1}{2}\zq^{*}C_{2}\bar{\zq}.
\end{cases}
\end{equation*}
where $m{\geq}2$ and $B_{1},B_{2},C_{1},C_{2}$
are $m\,{\times}\,m$ complex matrices,
with $C_{1}$ and $C_{2}$ symmetric.  
This is a $2m$-dimensional smooth real submanifold of $\C^{2+m}$. We keep the notation
$\Lambda_{B_{i},C_{i}}$ for the $\R$-linear maps $\C^{m}\,{\ni}\,\zq\,{\to}\,B_{i}\zq\,{+}\,C_{i}\bar{\zq}\,{\in}\,\C^{m}$
($i=1,2$) 
introduced at page\,\pageref{e7.11}. 
If $\Lambda_{B_{1},C_{1}}$ and $\Lambda_{B_{2},C_{2}}$ are not proportional, 
 the $CR$-regular points form a $CR$ manifold
$\sfM_{m-2,4}$ which is open and dense in $\sfM$  and  
\begin{equation*}
 \sfM=\sfM_{(m-2,4)}\cup\sfM_{(m-1,2)}\cup\sfM_{m,0},
\end{equation*}
its $CR$ singularities occurring at points where the four differentials 
\begin{equation*}
 d\wq^{1}-\zq^{*}B_{1}d\zq,\;\; d\wq^{2}-\zq^{*}B_{2}d\zq,\;\; (z^{*}B_{1}^{*}\,{+}\,\zq^{\intercal}\bar{C}_{1})d\zq,\;\; 
 (\zq^{*}B_{2}^{*}\,{+}\,\zq^{\intercal}\bar{C}_{2})d\zq
\end{equation*}
are linearly dependent:  $\sfM^{\sharp}\,{=}\, \sfM_{(m-1,2)}\,{\cup}\,\sfM_{m,0}$, with 
\begin{equation*} 
\begin{cases}
 \sfM_{(m,0)}=\{(\wq,\zq)\,{\in}\,\sfM\mid \Lambda_{B_{1},C_{1}}(\zq)\,{=}\,0,\; \Lambda_{B_{2},C_{2}}(\zq)\,{=}\,0\},\\
 \sfM_{(m-1,2)}=\left\{(\wq,\zq)\,{\in}\,\sfM\,{\setminus}\sfM_{(m,0)}\,\left|\, \begin{aligned} 
 \exists\,(\lambdaup^{1},\lambdaup^{2})\,{\in}\,\C^{2}{\setminus}\{0\} \;\;
 \text{s.t.}\qquad\\
 \lambdaup^{1}\Lambda_{B_{1},C_{1}}(\zq)\,{+}\,\lambdaup^{2}\Lambda_{B_{2},C_{2}}(\zq)\,{=}\,0
 \end{aligned}\right.\right\}.
  \end{cases}
\end{equation*}
If $\ker(\Lambda_{B_{1},C_{1}})\,{\cap}\,\ker(\Lambda_{B_{2},C_{2}})\,{=}\,\{0\}$, then $\sfM_{(m,0)}\,{=}\,\{0\}$.
Let us assume that $\Lambda_{B_{1},C_{1}}$ is nonsingular. Then 
\begin{equation*}
 \sfM_{(m-1,2)}=\{(\wq,\zq)\,{\in}\,\sfM{\setminus}\{0\}\mid \exists\,\lambdaup\,{\in}\,\C\;\text{s.t.}\;
 \Lambda_{B_{2},C_{2}}(\zq)=\lambdaup\,\Lambda_{B_{1},C_{1}}(\zq)\}.
\end{equation*}
In this case the condition that  
$\ker(\Lambda_{B_{2},C_{2}}{-}\,\lambdaup\,\Lambda_{B_{1},C_{1}})$
is, for every $\lambdaup\,{\in}\,\C$, either $0$, or a real or a complex line 
ensures the transversality of $CR$-singularities of type $(m{-}1,2)$. 
This condition can also be formulated by requiring that, for all complex $\lambdaup,$ the $\R$-linear
endomorphism $(\Lambda_{B_{2},C_{2}}{-}\,\lambdaup\,\Lambda_{B_{1},C_{1}})$ of $\C^{m}$
is \emph{semisimple}.
\par
Under the transversality conditions above, the $CR$-singular subset
$\sfM^{\sharp}$ has
transversal complex
dimension $1$. Thus, \par
\textsl{If all nontrivial real linear combinations of the hermitian symmetric matrices
$(B_{1}{+}B^{*}_{1}),\, i(B_{1}{-}B_{1}^{*}),
\,(B_{2}{+}B^{*}_{2}),\, i(B_{2}{-}B_{2}^{*})$ have at least $4$ non positive and at least 
$3$ positive and $3$ negative
eigenvalues, then $\sfM$ has the holomorphic extension property.}\par
We can take e.g. $m\,{=}\,8$ and 
\begin{equation*}
 B_{1}= 
\begin{pmatrix}
 0 & \Id_{4}\\
 0 & 0
\end{pmatrix}, \; C_{1}= 
\begin{pmatrix}
 0 & 0\\
 \Id_{4} & 0
\end{pmatrix}
,\; B_{2}= \left( 
\begin{smallmatrix}
 0&0&0&0&0&0&0&0\\
  0&0&0&0&0&0&0&1\\
   0&0&0&0&0&0&1&0\\
    0&0&0&0&0&1&0&0\\
     0&0&0&0&0&0&0&0\\
      0&0&0&0&0&0&0&0\\
       0&0&0&0&0&0&0&0\\
        0&0&0&0&0&0&0&0\\
\end{smallmatrix}
\right)
,\; C_{2}= 0.
\end{equation*}
\end{exam}
\par\smallskip
Weak pseudoconcavity may be lost under small perturbations. Thus in the following, to get
examples of smooth manifolds on which
maximum modulus principles and holomorphic
extendability are preserved also after small perturbations, 
we consider how strict pseudoconcavity is preserved near a $CR$ singularity.
\begin{lem}
 Let $\sfM$ be described by \eqref{e7.13} near $0$. Then, if $H_{\vec{c}}$ has at least $\qq$ positive
 and $\qq$ negative eigenvalues for all $\vec{c}\,{\neq}\,0$, with $p\,{<}\,\qq$, 
 then $\sfM$ is a strictly $(q{-}p)$-pseudoconcave
 $CR$ manifolds at all $CR$ regular points close to $0$. \end{lem} 
\begin{proof} It suffices to prove our contents for the quadric obtained from 
\eqref{e7.13} by erasing the $0(3)$ addenda. Indeed 
perturbations by defining functions vanishing to the second
order at $0$ specified above preserve 
strict
pseudoconcavity on a neighbourhood of $0$.
Take the real coordinates of page\,\pageref{l7.9}.
 We can use $x,y,\vq$ as real coordinates on a neighbourhood of $0$ in $\sfM$. At a point $(x_{0},y_{0},\vq_{0})$
 we obtain 
\begin{equation*} 
\begin{cases}
2\uq^{j}=(\zq{-}\zq_{0})^{*}H_{j}\zq_{0}+\zq_{0}^{*}H_{j}(\zq{-}\zq_{0})+(\zq{-}\zq_{0})^{*}
 H_{j}(\zq{-}\zq_{0}) \; 1{\leq}i{\leq}k{-}2,\\
2\rr^{\iota}= (\zq{-}\zq_{0})^{*}H_{h+2\iota-1}\zq_{0}+\zq_{0}^{*}H_{h+2\iota-1}(\zq{-}\zq_{0})+(\zq{-}\zq_{0})^{*}
 H_{h+2\iota-1}(\zq{-}\zq_{0})+\cdots,\\
2\sq^{\iota}= (\zq{-}\zq_{0})^{*}H_{h+2\iota}\zq_{0}+\zq_{0}^{*}H_{h+2\iota}(\zq{-}\zq_{0})+(\zq{-}\zq_{0})^{*}
 H_{h+2\iota}(\zq{-}\zq_{0})+\cdots,
\end{cases}
\end{equation*}
where the dots represent pluriharmonic summands. By restricting the Hermitian forms $H_{\vec{c}}$ to
a complex $n{-}k$-plane of 
$\C^{n-k+p}$ we may loose at most $p$ positive and $p$ negative eigenvalues. 
\end{proof} 

%%%%%%%%%%%%%%%%%%
To construct 
\textit{quadrics}
with properties of \textit{strict} pseudoconcavity,
we can use results from \cite{ad62,ALP65}.\par
A positive integer $n$ can be uniquely decomposed by 
\begin{equation}
 n=2^{a+4b}(2c+1),\;\; \text{with $a,b,c\,{\in}\,\Z_{+}$ and $1{\leq}a{\leq}3$.}
\end{equation}
Set 
\begin{equation} \rhoup_{\C}(n)
=2(a+4b)+2.
\end{equation}
A Hermitian symmetric $m\,{\times}\,m$ matrix can have at most $\left[\tfrac{m}{2}\right]$
eigenvalues of opposite signs. By \cite{ad62,ALP65} we have
\begin{lem}\label{l7.6}
 The maximal dimension of a linear space of $2\qq\,{\times}\,2\qq$ 
hermitian symmetric matrices
 having at least $\qq$ eigenvalues of each sign is 
$\rhoup_{\C}(\qq){+}1$.\qed
\end{lem}
Bu using Lemma\,\ref{l7.6} one can construct several \textit{quadrics} with the correct 
amount of pseudoconavity for any codimension, providing examples where holomorphic
extendability holds at transversal $CR$ singularities.  
\begin{exam}
 The Clifford algebra generated by $4$ anti-commuting imaginary units is the algebra $\Hb(2)$ of
 $2\,{\times}\,2$ matrices of quaternions. This can also be represented as a subalgebra of the
 algebra $\C(4)$ of complex $4\,{\times}\,4$-matrices. In particular, the imaginary units can
 be represented as  complex $4\times{4}$ invertble matrices $Q_{1},Q_{2},Q_{3},Q_{4}$. \par
 We can take, e.g. 
\begin{gather*}
 Q_{1}= 
\begin{pmatrix}
 0 & 0 &1& 0\\
 0 & 0 &0 & 1\\
 {-}1& 0 & 0 & 0\\
 0 & {-}1 & 0 & 0
\end{pmatrix},\;
\qquad Q_{2}= 
\begin{pmatrix}
 0 & 0 &i& 0\\
 0 & 0 &0 & -i\\
 i& 0 & 0 & 0\\
 0 & {-}i & 0 & 0
\end{pmatrix},\\
 Q_{3}= 
\begin{pmatrix}
 0 & 0 &0& 1\\
 0 & 0 &{-}1 &0 \\
 0& 1 & 0 & 0\\
 {-}1& 0 & 0 & 0
\end{pmatrix},\;
\qquad Q_{4}= 
\begin{pmatrix}
 0 & 0 &0& i\\
 0 & 0 &i &0 \\
 0& i & 0 & 0\\
 i& 0 & 0 & 0
\end{pmatrix}
\end{gather*}
 Then all nonzero matrices in the real span of 
\begin{align*}
 H_{1}= 
\begin{pmatrix}
 0 & Q_{1}\\
 Q_{1}^{*}& 0
\end{pmatrix},\; H_{2}= 
\begin{pmatrix}
 0 & Q_{2}\\
 Q_{2}^{*}& 0
\end{pmatrix},\; H_{3}= 
\begin{pmatrix}
 0 & Q_{3}\\
 Q_{3}^{*}& 0
\end{pmatrix},\; H_{4}= 
\begin{pmatrix}
 0 & Q_{4}\\
 Q_{4}^{*}& 0
\end{pmatrix}\\
H_{5}= 
\begin{pmatrix}
 \Id_{4}& 0\\
 0 & {-}\Id_{4}\;
\end{pmatrix}
\end{align*}
are Hermitian symmetric with $4$ positive and $4$ negative eigenvalues and 
\begin{equation*}
 \sfM\,:\, 
\begin{cases}
 \wq^{1}+\bar{\wq}^{1}=\zq^{*}H_{1}\zq,\\
 \wq^{2}+\bar{\wq}^{2}=\zq^{*}H_{2}\zq,\\
  \wq^{3}+\bar{\wq}^{3}=\zq^{*}H_{3}\zq,\\
 \tq=\zq^{*}(H_{5}+i\,H_{5})\zq,\\
\end{cases}\quad \zq\in\C^{8},\; \wq^{1},\wq^{2},\wq^{3},\tq\in\C
\end{equation*}
is a quadric of $\C^{12}$ having a $3$-dimensional smooth submanifold $\sfM_{8,3}$ of
$CR$ singularities and 
the holomorphic extension property at all points, because its open dense set $\sfM_{7,5}$ is 
a strictly $3$-pseudoconcave $CR$ manifold. 
If $\psiup_{1},\psiup_{2},\psiup_{3},\psiup_{4},\psiup_{5}$ are smooth real-valued functions with compact
support in $\C^{n}$, then, for small real $\epsilonup$, also 
\begin{equation*}
 \sfM_{\epsilonup}\,:\, 
\begin{cases}
 \wq^{1}+\bar{\wq}^{1}=\zq^{*}H_{1}\zq+\epsilonup\psiup_{1},\\
 \wq^{2}+\bar{\wq}^{2}=\zq^{*}H_{2}\zq+\epsilonup_{2}\psiup_{2},\\
  \wq^{3}+\bar{\wq}^{3}=\zq^{*}H_{3}\zq+\epsilonup_{2}\psiup_{3},\\
 \tq=\zq^{*}(H_{4}+i\,H_{5})\zq+\epsilonup(\psiup_{4}+i\psiup_{5}),\\
\end{cases}\quad \zq\in\C^{8},\; \wq^{1},\wq^{2},\wq^{3},\tq\in\C
\end{equation*}
is,  for small real $\epsilonup$, a smooth real submanifold of $\C^{12}$, having a 
$3$-dimensional real sumbanifold $\sfM_{\epsilonup}^{\sharp}$ of $CR$ singularities of type $(8,3)$
and the holomorphic extension property.
\begin{exam}
 Let us consider the $3\,{\times}\,4$ matrices 
\begin{equation*}
 R_{1}= 
\begin{pmatrix}
 1 & 0 & 0 & 0 \\
 0 & 1 & 0 & 0 \\
 0 & 0 & 1 & 0 \\
 \end{pmatrix}, \quad R_{2}= 
\begin{pmatrix}
 0 & 1 & 0 & 0 \\
 0 & 0 & 1 & 0 \\
 0 & 0 & 0 & 1 \\
\end{pmatrix}
\end{equation*}
and the $7\,{\times}\,7$ Hermitian matrices 
\begin{equation*}
 H_{1}= 
\begin{pmatrix}
 0 & R_{1}\\
 R_{1}^{*}& 0
\end{pmatrix},\;\; H_{2}= 
\begin{pmatrix}
 0 & R_{2}\\
 R_{2}^{*}& 0
\end{pmatrix},\;\; H_{3}= 
\begin{pmatrix}
 \Id_{3} & 0\\
 0 & {-}\Id_{4}
\end{pmatrix}.
\end{equation*}
All their linear combinations have at least $3$ positive and $3$ negative eigenvalues. Thus, if
$\psiup_{1},\psiup_{2},\psiup_{3}$ are real-valued smooth functions with compact support in $\C^{9}\,{=}\,
\C^{7}_{\zq}\,{\times}\,\C^{2}_{\wq}$, the smooth real submanifold 
\begin{equation*}
 \sfM\,:\quad 
\begin{cases}
 \wq+\bar{\wq}=\zq^{*}H_{1}\zq+\epsilonup\psiup_{1},\\
 \tq\,{=}\,\tfrac{1}{2}\zq^{*}(H_{2}\,{+}\,i\,H_{3})\zq+4\zq^{*}\bar{\zq}+\epsilonup(\psiup_{2}\,{+}\,i\,\psiup_{3})
\end{cases}
\end{equation*} is smooth, 
of real codimension $3$ in $\C^{9}$, has a $1$-dimensional smooth singular locus $\sfM_{(7,1)}$
and, being strictly $2$-pseudoconcave at all points of $\sfM_{(6,3)}$, has the holomorphic extension property.
\end{exam}
\par
When $D\,{=}\,0$, 
the set $\sfM^{\sharp}$ of
$CR$-singular points of type $(n{-}k{+}1,k{-}2)$ of the quadric \eqref{e7.7}
is parametrized by the kernel of $\Lambda_{B,C,0}$, which is 
the cartesian product of $\R^{h}$ times the
kernel 
of the $\R$-linear map  $\Lambda_{B,C}$ of \eqref{e7.11}.
This is a real lineaR subspace of $\C^{m}$ 
and thus a $CR$ submanifold of type $(r_{B,C},s_{B,C})$ for a pair
on nonnegative integers $r_{B,C},s_{B,C}$ 
with $2r_{B,C}\,{+}\,s_{B,C}\,{\leq}\,m$. Its lifting $\sfM^{\sharp}$ is 
a $CR$ submanifold of type $(r_{B,C},s_{B,C}{+}h)$ and hence has transversal complex dimension
$\ell\,{=}\,\left(r_{B,C}{+}\left[\tfrac{s_{B,C}{+}h{+}1}{2}\right]\right)$ (see Example\,\ref{e6.6}). \par\smallskip
Let a slight modification of \eqref{e7.7} be described by
\begin{equation}\label{e7.19}
 \sfM\,:\, 
\begin{cases}
 \wq^{1}\,{+}\,\bar{\wq}^{1}=\zq^{*}H_{1}\zq+\epsilonup\psiup_{1}(\zq,\bar{\zq},\vq),\\
 \hdots\hdots\hdots\hdots\\
  \wq^{h}\,{+}\,\bar{\wq}^{h}=\zq^{*}H_{h}\zq+\epsilonup\psiup_{h}(\zq,\bar{\zq},\vq),\\
 \tq=\tfrac{1}{2}\zq^{*}B\zq+\tfrac{1}{2}\zq^{*}C\,\bar{\zq}
\end{cases}\end{equation}
for smooth real-valued $\psiup_{1},\hdots,\psiup_{h}$, vanishing to the second order at $0$.
\par
\smallskip
\textsl{
 Let $\sfM$ be defined by \eqref{e7.19}. 
 Then 
 there is an open neighbourhood $U$ of $0$ in $\sfM$ such that 
 all $CR$ singularities of $\sfM$ in $U$ are of type $(m,h)$ and 
$U\,{\cap}\,\sfM_{(m,h)}$ is contained in a $CR$ submanifold of type $(r_{B,C},h{+}s_{B,C})$ of $\sfM$.}
\par
\textsl{If $\sfM_{(m-1,h+2)}$ is strictly $\qq$-pseudoconcave with $q\,{>}\,\ell$, then $\sfM$ has the
holomorphic extension property.}\par\smallskip
We note that a small perturbation of the last equation in \eqref{e7.19} would make the $CR$ singularity
at $0$ \textit{transversal}, reducing the locus of $CR$ singularities 
 to a totally real $h$ dimensional submanifold.
\end{exam} 

\bibliographystyle{amsplain}
\renewcommand{\MR}[1]{}
\bibliography{cr}

\end{document}